\documentclass{amsart}
\usepackage{latexsym,amsxtra,amscd,ifthen,amsmath,color,hyperref}
\usepackage{enumerate}
\usepackage{amsfonts}
\usepackage{verbatim}
\usepackage{amsmath}
\usepackage{amsthm}
\usepackage{amssymb}
\usepackage[notcite,notref,final]{showkeys}
 \usepackage[all,cmtip]{xy}

\numberwithin{equation}{section}

\theoremstyle{plain}
\newtheorem{theorem}{Theorem}[section]
\newtheorem{lemma}[theorem]{Lemma}

\newtheorem{proposition}[theorem]{Proposition}
\newtheorem{hypothesis}[theorem]{Hypothesis}
\newtheorem{corollary}[theorem]{Corollary}
\newtheorem{conjecture}[theorem]{Conjecture}

\theoremstyle{definition}
\newtheorem{definition}[theorem]{Definition}
\newtheorem{example}[theorem]{Example}

\newtheorem{notation}[theorem]{Notation}
\newtheorem{definitiontheorem}[theorem]{Definition-Theorem}

\newtheorem{remark}[theorem]{Remark}
\newtheorem{question}[theorem]{Question}

\makeatletter              
\let\c@equation\c@theorem  
\makeatother

\DeclareMathOperator{\hdet}{hdet}

\DeclareMathOperator{\Ext}{Ext}

\DeclareMathOperator{\injdim}{injdim}

\DeclareMathOperator{\GKdim}{GKdim}
\DeclareMathOperator{\End}{End}

\DeclareMathOperator{\Hom}{Hom}

\DeclareMathOperator{\im}{im}

\newcommand{\fm}{\mathfrak{m}}

\newcommand{\cal}{\mathcal}
\newcommand{\mf}{\mathfrak}
\newcommand{\mc}{\mathcal}

\newcommand{\assign}{:=}
\newcommand{\nin}{\not\in}

\begin{document}

\title[Quantum binary polyhedral groups]
{Quantum binary polyhedral groups\\
and their actions on quantum planes}

\author{K. Chan, E. Kirkman, C. Walton and J.J. Zhang}

\address{Chan: Department of Mathematics, Box 354350,
University of Washington, Seattle, Washington 98195,
USA}

\email{kenhchan@math.washington.edu}

\address{Kirkman: Department of Mathematics,
P. O. Box 7388, Wake Forest University,
Winston-Salem, NC 27109, USA}

\email{kirkman@wfu.edu}

\address{Walton: Department of Mathematics, Massachusetts 
Institute of Technology, Cambridge, Massachusetts 02139,
USA}

\email{notlaw@math.mit.edu}

\address{zhang: Department of Mathematics, Box 354350,
University of Washington, Seattle, Washington 98195,
USA}

\email{zhang@math.washington.edu}

\bibliographystyle{abbrv}       

\begin{abstract}
We classify quantum analogues of actions of finite subgroups $G$ of 
$SL_2(k)$ on commutative polynomial rings $k[u,v]$. More precisely, 
we produce a classification of pairs $(H, R)$, where $H$ is a finite 
dimensional Hopf algebra that acts inner faithfully and preserves 
the grading of an Artin-Schelter regular algebra $R$ of global dimension 
two. Remarkably, the corresponding invariant rings $R^H$ share similar 
regularity and Gorenstein properties as the invariant rings $k[u,v]^G$ 
in the classical setting. We also present several questions and directions 
for expanding this work in noncommutative invariant theory.
\end{abstract}



\maketitle

\tableofcontents

\setcounter{section}{-1}


\section{Introduction}
\label{sec0}

Let $k$ be an algebraically closed field of characteristic zero, unless 
stated otherwise. This work is part of a 
program which extends classical invariant theory to a noncommutative setting.
We depart from the traditional situation of
finite groups acting linearly on $k[u,v]$ by considering finite
dimensional Hopf algebras acting on noncommutative analogues of $k[u,v]$.
The latter algebras are called {\em{Artin-Schelter (AS)}} regular algebras of global dimension $2$.
Previous work \cite{JorgensenZhang}, \cite{KKZ:Rigidity},
\cite{KKZ:Gorenstein}, \cite{KKZ:STC} demonstrates that there is a 
rich invariant theory in this context. 

The goal of this paper is to classify noncommutative analogues of linear 
actions of finite subgroups of $SL_2(k)$ 
on AS regular algebras of global dimension $2$ and study the resulting rings of invariants. 
The {\em{homological determinant}} 
generalizes the determinant associated to a linear group action \cite{JingZhang} 
\cite{KKZ:Gorenstein},
and so we only consider actions with trivial homological determinant. 
With this assumption, the rings of invariants turn out to have good homological properties
(AS Gorenstein). In fact, they are often isomorphic to the coordinate rings of Kleinian singularities.
The finite subgroups of $SL_2(k)$ were classified by Felix Klein, and their 
invariant subrings play an important role in classical invariant theory, 
representation theory, and algebraic geometry. Understanding their 
noncommutative analogues is an important contribution to noncommutative 
invariant theory.

Let us discuss in more detail the noncommutative structures mentioned 
above. Naturally, one can view quantum analogues of finite subgroups 
of $SL_2(k)$ (or {\it quantum binary polyhedral groups}) as finite 
subgroups of the quantum group $SL_q(2)$ for $q\in k^{\times}$. Following 
Drinfeld,  we interpret finite subgroups of $SL_q(2)$ as finite dimensional
Hopf quotients of the coordinate Hopf algebra 
$\mc{O}_q(SL_2(k))$. The problem of classifying finite dimensional Hopf quotients 
of $\mc{O}_q(SL_2(k))$ has been examined thoroughly in the literature 
(see e.g.\cite{BichonNatale}, \cite{Muller}, \cite{Stefan}).
The novelty of our work is that we address the 
question of when these non(co)commutative finite subgroups of $SL_2(k)$ 
act on AS regular algebras $R$ of global dimension 
$2$.

Since we assumed $k$ is algebraically closed, the AS regular algebras $R$ of 
global dimension 2, generated in degree one, are isomorphic to either
$$k_J[u,v]:=k\langle u,v\rangle/ (vu-uv-u^2), \qquad
{\text{or}}\qquad k_q[u,v]:=k\langle u,v\rangle/(vu-quv),$$ where $q\in
k^{\times}$ [Example~\ref{ex1.2}]. These algebras are also known as 
{\it quantum planes}.

We say that a finite dimensional Hopf algebra $H$ acts on an algebra 
$R$ if $R$ is a {\it left $H$-module algebra} [Definition \ref{def1.3}]. 
The following statements are \underline{standing assumptions} for this article.

\begin{hypothesis}
\label{hyp0.1} 
\begin{enumerate} 
\item[(1)]
Let $H\neq k$ be a finite-dimensional Hopf algebra.
\item[(2)]
Let $R$ be an Artin-Schelter (AS) regular 
algebra of global dimension 2 that is generated in degree one, that is,
$R$ is isomorphic to either $k_J[u,v]$ or $k_q[u,v]$.
\item[(3)]
Let $H$ act on $R$ inner faithfully  [Definition~\ref{def1.5}], while 
preserving the
grading of $R$.
\end{enumerate}
\end{hypothesis}

The hypotheses below are \underline{used selectively} throughout this work.

\begin{hypothesis}
\label{hyp0.2} 
Assume Hypothesis \ref{hyp0.1} with the additional condition:
\begin{enumerate}
\item[(4)]
The $H$-action on $R$ has trivial homological determinant
[Definition \ref{def1.7}].
\end{enumerate}
\end{hypothesis}

\begin{hypothesis}
\label{hyp0.3} 
Assume Hypothesis \ref{hyp0.2} with the additional condition:
\begin{enumerate}
\item[(5)]
The Hopf algebra $H$ (with antipode $S$) is semisimple, and hence 
$S^2=Id_H$ \cite[Theorem 3]{LarsonRadford}.
\end{enumerate}
\end{hypothesis}

Conditions (1)-(3) generalize faithful, linear actions of  finite
groups $G$ on $k[u,v]$ to the quantum setting. We view condition (4) as
analogous to the condition that $G\subseteq SL_2(k)$ in the following
sense.  In the classical setting $G$ is realized as a subgroup
of $GL_2(k)$, and the usual determinant induces a group homomorphism,
{\sf det}$:G\hookrightarrow GL_2(k) \rightarrow k^{\times}.$
Note that det is not intrinsic to $G$, but depends on the action of
$G$ on $k[u,v]$. The {\it homological determinant} of
a Hopf algebra action on a graded $k$-algebra $R$ is a useful
generalization of the map det.  Although not needed for all our results, 
semisimplicity (Condition (5)) will be used for a large part of this paper, in particular 
for our results on the regularity and the geometry of the resulting 
invariant rings $R^H$ [Theorem~\ref{thm0.6} and Proposition~\ref{thm0.7}].

Now we state our main theorem. We say that a Hopf algebra is 
{\it nontrivial} if it is both noncommutative and noncocommutative.

\begin{theorem}[Lemma~\ref{lem4.1}, Theorems~\ref{thm4.5},
\ref{thm5.2}, \ref{thm6.2}]
\label{thm0.4} 
Let $R$ be an Artin-Schelter regular algebra of global dimension $2$. 
The finite dimensional Hopf algebras $H$ which act inner faithfully on 
$R$,  and satisfy Hypotheses~\ref{hyp0.2}, are
classified in Table 1.  In particular, we have the following statements.
\begin{enumerate}
\item
If $R$ is commutative, then $H$ is cocommutative.
\item
If $R$ is non-PI, then $H$ is both commutative and cocommutative.
\item
If $R$ is noncommutative and PI, then there are nontrivial Hopf algebras
acting inner faithfully on $R$.
\end{enumerate}
Furthermore, if $H$ is non-semisimple, then $R\cong k_q[u,v]$ where $q$
is an $n$-th root of unity for $n\geqslant 3$.
\end{theorem}

\noindent {\it Notation.} [$\tilde{\Gamma}$, $\Gamma$, $C_n$, $D_{2n}$, $U$] 
Let $\tilde{\Gamma}$ denote a finite subgroup of $SL_2(k)$, $\Gamma$ 
denote a finite subgroup of $PSL_2(k)$, $C_n$ denote a cyclic group of 
order $n$, and $D_{2n}$ denote a dihedral group of order $2n$. Let 
$\text{ord}(q)$ denote the order of $q$, for $q \in k^{\times}$ a 
root of unity. We also write $R=k\langle U \rangle / I$ where $U=ku\oplus kv$
and $I$ is the two sided ideal generated by the relation.

\[
\begin{array}{|l|l|l|}
\hline
\text{AS regular alg. $R$ of gldim 2} & \text{f.d. Hopf algebra(s) $H$
acting on $R$} & \text{Result number(s)}\\
\hline
\hline
&&\\
k[u,v] & k\tilde{\Gamma}
& \text{\ref{thm4.5}(a1), \ref{thm5.2}(b1,b3)}\\
\hline
k_{-1}[u,v] & kC_n~\text{for}~n\geq 2  
& \text{\ref{thm5.2}(b1,b2,b3)}\\
& kD_{2n} & \text{\ref{thm4.5}(a2)}\\
&  (kD_{2n})^{\circ}
& \text{\ref{thm5.2}(b4)}\\
&\mc{D}(\tilde{\Gamma})^{\circ}, \tilde{\Gamma} \text{ nonabelian}& \text{\ref{thm4.5}(a3) and \ref{cor4.6}}\\
\hline
k_q[u,v], ~q \text{ root of 1, $q^2 \neq$ 1} &&\\
{\text{ if $U$ non-simple}} & kC_n \text{~for~} n\geq 2
& \text{\ref{thm5.2}(b1,b3)}\\
& (T_{q, \alpha, n})^{\circ} &\text{\ref{thm6.2}(c1)}\\
\text{ if $U$ simple, ord($q$)  odd} & H \text{ in } 1 \to 
(k\tilde{\Gamma})^{\circ} \to H^{\circ}
\to \mf{u}_{q}(\mf{sl}_2)^{\circ} \to 1; & \text{\ref{thm6.2}(c2);}\\
\text{ if $U$ simple, ord($q$) even,}~ q^4 \neq 1
& H \text{ in } 1 \to (k\Gamma)^{\circ} \to H^{\circ}
\to \mf{u}_{2,q}(\mf{sl}_2)^{\circ} \to 1; & \text{\ref{thm6.2}(c3) and \ref{pro6.25};}\\
\text{ if $U$ simple}, ~q^4 =1
& 
\begin{tabular}{l}
\hspace{-.07in}$H \text{ in } 1 \to (k\Gamma)^{\circ} \to H^{\circ}
\to \mf{u}_{2,q}(\mf{sl}_2)^{\circ} \to 1 ~~ \text{or}$  \\
$  \phantom{H \text{ in }} 1 \to (k\Gamma)^{\circ} \to H^{\circ}
\to  \frac{\mf{u}_{2,q}(\mf{sl}_2)^{\circ}}{(e_{12}-e_{21} e_{11}^2)}\to 1$
\end{tabular}
& \text{\ref{thm6.2}(c3) and \ref{rem6.24}}\\
\hline
k_q[u,v], ~q \text{ not a root of 1} & kC_n, n \geq 2& 
\text{\ref{thm5.2}(b1,b3)}\\
\hline
k_J[u,v] & kC_2 & \text{\ref{thm5.2}(b1)}\\
\hline
\end{array}
\]
\smallskip

\begin{center}
Table 1: Summary of Theorem \ref{thm0.4}
\end{center}
\bigskip

When $U$ is a simple left $H$-module, we can fit $H$ into
the following diagram, where the rows are exact sequences of Hopf algebras.
$$
\xymatrix@-1pc{
k \ar[r] & \mc{O}(G) \ar[r]  \ar@{->>}[d]
& \mc{O}_q(SL_2(k)) \ar[r] \ar@{->>}[d]
& \overline{\mc{O}_q(SL_2(k))} \ar[r] \ar@{->>}[d] & k\\
k \ar[r] &  (kG')^{\circ} \ar[r] & H^{\circ} \ar[r] & K \ar[r] & k
}
$$

\medskip

\[
\begin{array}{|c|c|c|c|c|}
\hline
 \text{Theorem no.}    & q & G & G' & K \\
\hline
\hline
\text{\ref{thm4.5}(a1)} & 1& SL_2(k) & \tilde{\Gamma} \text{~nonabelian} & k\\
\text{\ref{thm4.5}(a2)} &-1& PSL_2(k)& C_n &  k C_2\\
\text{\ref{thm4.5}(a3)} &-1& PSL_2(k)& \Gamma  \text{~nonabelian} & k C_2\\
\text{\ref{thm6.2}(c2)} & \text{ord}(q) \text{ is odd} & SL_2(k)
& \tilde{\Gamma}
& {\tiny \overline{\mc{O}_q(SL_2(k))} \cong \mf{u}_q(\mf{sl}_2)}^{\circ}\\
\text{\ref{thm6.2}(c3)} & \text{ord}(q) \text{ is even,}  ~q^4 \neq 1& PSL_2(k)
& \Gamma
& {\tiny \overline{\mc{O}_q(SL_2(k))} \cong \mf{u}_{2,q}(\mf{sl}_2)}^{\circ}\\
\text{\ref{thm6.2}(c3)} & \text{ord}(q) \text{ is even,}  ~q^4 = 1& PSL_2(k)
& \Gamma
& \begin{tabular}{rl}
${\tiny \overline{\mc{O}_q(SL_2(k))}} \cong$ &
${\tiny \mf{u}_{2,q}(\mf{sl}_2)}^{\circ}$\\
& $\text{or } 
{\tiny \frac{\mf{u}_{2,q}(\mf{sl}_2)^{\circ}}{(e_{12}-e_{21}e_{11}^2)}}$
\end{tabular}\\
\hline
\end{array}
\]
\smallskip

\begin{center}
Diagram-Table 2: For
$H$-actions with $U$ a simple $H$-module
\end{center}

\bigskip

Note that part (a) of Theorem~\ref{thm0.4} is proved in both
\cite[Proposition 0.7]{CWZ:Nakayama} and \cite[Theorem~1.3]{EtingofWalton:ssHopf} via different techniques. 
Moreover, the 
most interesting (nontrivial) Hopf algebra actions occur on the PI 
algebras $k_q[u,v]$ for $q$ a root of unity. 

Next, we study the rings of invariants arising from the Hopf algebra 
actions appearing in Theorem \ref{thm0.4}, again with the aim of 
generalizing results  on actions of finite subgroups of $SL_2(k)$ on 
$k[u,v]$ in classical invariant theory. Watanabe's theorem 
implies that if $G$ is a
finite subgroup of $SL_2(k)$, then $k[u,v]^G$ is Gorenstein 
\cite[Theorem 1]{Watanabe}. By \cite[Theorem 0.1]{KKZ:Gorenstein}, 
if the pair $(H,R)$ satisfies Hypothesis \ref{hyp0.3} (with $H$ 
semisimple), then $R^H$ is {\it Artin-Schelter Gorenstein} 
[Definition~\ref{def1.1}]; we extend this result to the non-semisimple case.

\begin{proposition} [\cite{KKZ:Gorenstein}, 
Propositions \ref{pro6.8}, \ref{pro6.14}, \ref{pro6.28}]
\label{pro0.5} Let $(H,R)$ be a pair as in Theorem \ref{thm0.4}. Then
the invariant subring $R^H$ is Artin-Schelter Gorenstein.
\end{proposition}

The Shepherd-Todd-Chevalley theorem 
implies $k[v_1, \dots, v_n]^G$ is not regular if $G$ acts via a subgroup of $SL_n(k)$.
Our next 
result gives a sufficient condition for the invariant subrings from 
Proposition \ref{pro0.5} \underline{not} to be Artin-Schelter regular.
Note that this result holds for actions on AS regular algebras of arbitrary
global dimension.

\begin{theorem}[Theorem \ref{thm2.3}]
\label{thm0.6}  
Let $H$ be a semisimple Hopf algebra, and $R$ be a noetherian connected 
graded Artin-Schelter regular algebra equipped with an $H$-module 
algebra structure. If $R^H\neq R$ and the homological determinant of 
the $H$-action on $R$ is trivial, then $R^H$ is not Artin-Schelter regular.
\end{theorem}

\noindent This theorem fails if $H$ is non-semisimple; see 
Lemma~\ref{lem6.7} (Case: $l=m=n$).

Finally, we consider the McKay correspondence in the context of Hopf 
algebra actions on Artin-Schelter regular algebras. As in the classical 
setting, the McKay quiver of the $H$-action on $R$ has vertices
indexed by the isomorphism classes of irreducible 
representations of $H$, and hence records information about the 
representations of $H$. The McKay quivers of the actions in 
Theorem~\ref{thm0.4} are given in the final result.

\begin{proposition}[Proposition~\ref{pro7.1}]
\label{thm0.7}  
Assume Hypothesis \ref{hyp0.3} for  a semisimple Hopf algebra $H$ acting 
on an Artin-Schelter regular algebra $R = k \langle U \rangle/(r)$ of
global dimension two, with trivial homological determinant.  Then, the 
McKay quiver $Q(H,U)$ is either of type A, D, E, L, or DL.
\end{proposition}

Classically, only types  A, D, E appear for the McKay quivers of finite
subgroups of $SL_2(k)$ (acting on $k[u,v]$). In particular, the 
McKay quiver of type DL comes from the dihedral group action on 
$k_{-1}[u,v]$ in Theorem~\ref{thm0.4}. However, in the classical setting, 
the trivial determinant condition precludes dihedral group actions. Hence, 
by replacing the algebra $k[u,v]$ with a noncommutative algebra, and by 
extending the notion of determinant, our noncommutative version of the 
McKay correspondence involves more groups. Further study of this 
correspondence is the subject of future work.

This paper is organized as follows. We provide  background material on 
Artin-Schelter regular algebras, Hopf algebra actions, and the 
homological determinant in Section \ref{sec1}. As trivial homological 
determinant is a vital condition in our work, Section 
\ref{sec2} is devoted to the study of Hopf actions on Artin-Schelter regular 
algebras under this condition. For example, we show how to detect if 
the homological determinant  of an $H$-action on an Artin-Schelter 
regular algebra of dimension 2 is trivial [Theorem \ref{thm2.1}].
In Section~\ref{sec3}, we prove some elementary results on self-dual $H$-modules
which are used later in the paper. Sections~\ref{sec4}, \ref{sec5}, and 
\ref{sec6} are dedicated to the proof of Theorem~\ref{thm0.4} in the 
cases where $H$ is semisimple and noncommutative, $H$ is commutative 
(so, semisimple), and $H$ is non-semisimple, respectively. Moreover, we 
compute the McKay quivers of our classified semisimple Hopf algebra 
actions in Section~\ref{sec:McKay}.  In Section~\ref{sec8}, we suggest 
several questions for further study.


\section{Background material}
\label{sec1}

Here, we provide background material for Artin-Schelter regular 
algebras, for Hopf algebra actions on graded algebras, and for the 
homological determinant of such actions.

\subsection{Artin-Schelter regularity}
\label{ssec1.1}

An algebra $R$ is said to be {\it connected graded} if
$R = k \oplus R_1 \oplus R_2 \oplus \cdots$
with $R_i \cdot R_j \subseteq R_{i+j}$ for
all $i,j \in \mathbb{N}$. The {\it Hilbert series} of $R$ is defined
to be $\sum_{i \in \mathbb{N}} (\dim_k R_i) t^i$. As mentioned in
the introduction, we consider a class of (noncommutative) graded
algebras that serve as noncommutative analogues of commutative
polynomial rings. These algebras are defined as follows.

\begin{definition}
\label{def1.1}
Let $R$ be a connected graded algebra. Then, $R$ is
{\it Artin-Schelter (AS) regular} if it satisfies the
conditions below:
\begin{enumerate}[(1)]
\item $R$ has finite global dimension;
\item $R$ has finite Gelfand-Kirillov dimension;
\item $R$ is {\it Artin-Schelter Gorenstein}, that is
\begin{enumerate}[(a)]
\item $R$ has finite injective dimension $d < \infty$, and
\item Ext$^i_R(k,R) = \delta_{i,d} \cdot k(\ell)$ for some
$\ell \in \mathbb{Z}$ called the AS index of $R$. Here,  $k(\ell)$ is the $\ell$-th shift.
\end{enumerate}
\end{enumerate}
\end{definition}

\begin{example}
\label{ex1.2} Since $k$ is algebraically closed, the AS regular 
algebras of global dimension two that are generated in degree 
one are listed below (up to isomorphism):
\begin{enumerate}[(i)]
\item
the {\it Jordan plane}: $k_J[u,v] := k \langle u,v \rangle/ (vu - uv - u^2)$,
and
\item the {\it skew polynomial ring}:
$k_q[u,v] := k \langle u,v \rangle/(vu - quv)$ for $q \in k^{\times}$.
\end{enumerate}
\end{example}

\subsection{Hopf actions and homological determinant}
\label{ssec1.2}

We adopt the usual notation for the Hopf structure of a Hopf algebra
$H$, namely $(H, m, u, \Delta, \epsilon, S)$. We also adopt Sweedler's notation for the comultiplication: $\Delta(h) =
\sum h_1 \otimes h_2$ for $h \in H$. Note that every 
finite dimensional Hopf algebra has bijective antipode, so the Hopf algebras
appearing in this paper all have bijective antipodes. 
Moreover, let
$H^{\circ}$ denote its {\it Hopf dual}, which is just the $k$-linear dual $H^{\ast}$ of $H$ since $H$ is 
finite dimensional. 
Consider a sequence of Hopf algebra maps
$$k \rightarrow L \overset{\iota}{\rightarrow} H
\overset{\pi}{\rightarrow} \bar{H} \rightarrow k$$
where $\iota$ is injective and $\pi$ is surjective. We say the sequence is {\it exact} 
if either of the following equivalent conditions hold:
\begin{enumerate}[(i)]
\item $\ker \pi$ = $HL^+$, where $L^+$=ker $\epsilon$ is the augmentation
ideal;
\item $L = H^{\mathrm{co} \pi} = \{h \in H ~|~ (\pi \otimes \text{id}) \Delta(h)
= 1 \otimes h\}$.
\end{enumerate}

\noindent
We say that $H'$ is a {\it normal} Hopf subalgebra of $H$ if
\begin{center}
ad$_l(h)(f) := \sum h_{1} f S(h_{2}) \in H'$ \hspace{.2in} and
\hspace{.2in} ad$_r(h)(f) := \sum S(h_{1}) f h_{2} \in H'$
\end{center}
for all $f \in H'$ and $h \in H$.
The above two conditions are equivalent since we assume that Hopf
algebras have bijective antipode.

Given a left $H$-module $M$, we denote the $H$-action by $\cdot : H
\otimes M \rightarrow M$. Similarly for a Hopf algebra $K$, given 
a right $K$-comodule $M$, we denote the $K$-coaction by 
$\rho: M \rightarrow M \otimes K$. If $H$ is finite dimensional, 
then $M$ is a left $H$-module if and only if $M$ is a right 
$H^{\circ}$-comodule.

Now, we recall basic facts about Hopf algebra actions; refer to
\cite{Montgomery} for further details.

\begin{definition}
\label{def1.3}
Let $H$ be a Hopf algebra $H$ and $R$ be a $k$-algebra. We say
that {\it $H$ acts on $R$} (from the left), or $R$ is a {\it left
$H$-module algebra}, if $R$ is a left $H$-module,
if $h \cdot (ab) = \sum (h_1 \cdot a)(h_2 \cdot b)$,
 and if $h \cdot 1_R = \epsilon(h)
1_R$, for all $h \in H$, and for all $a,b \in R$.  

The {\it invariant subring} of
such an action is defined to be
$$R^H = \{ a \in R ~|~ h \cdot a = \epsilon(h) a, ~\forall h \in H \}.$$

Dually, we say that a Hopf algebra $K$ {\it coacts on} $R$ (from the right),
or $R$ is a {\it right $K$-comodule algebra}, if $R$ is a right 
$K$-comodule, if $\rho(1_R) = 1_R \otimes 1_K$,  and 
if $\rho(ab) = \rho(a) \rho(b)$, for all $a,b \in R$. The
{\it coinvariant subring} of such a coaction is given as follows:
$$R^{\mathrm{co} K} = \{a \in R ~|~ \rho(a) = a \otimes 1_K\}.$$
\end{definition}

Here is an example of a Hopf coaction on an AS regular algebra, which 
will be used in Sections~\ref{sec4} and~\ref{sec6}.

\begin{example}  
\label{ex1.4} 
Consider the {\it quantum special linear group} $\mc{O}_q(SL_2(k))$ 
for $q \in k^{\times}$, generated by 
$e_{11}$, $e_{12}$, $e_{21}$, $e_{22}$, subject to relations:
\[
\begin{array}{lll}
e_{12}e_{11} = q e_{11} e_{12}, & e_{21}e_{11} = q e_{11} e_{21}, 
\hspace{.2in}
e_{22}e_{12} = q e_{12} e_{22}, & 
\hspace{.15in} e_{22}e_{21} = q e_{21} e_{22},\\
e_{21}e_{12} = e_{12} e_{21}, & e_{22}e_{11} 
= e_{11} e_{22} + (q - q^{-1})e_{12}e_{21}, & 
\hspace{.15in} e_{11}e_{22} - q^{-1}e_{12} e_{21} 
= e_{22} e_{11} - q e_{12}e_{21}=1.
\end{array}
\]
The coalgebra structure and antipode are given by 
$\Delta(e_{ij}) = \sum_{m=1}^2 e_{im} \otimes e_{mj}$, 
$\epsilon(e_{ij}) = \delta_{ij}$, and
$$S(e_{11}) = e_{22}, \hspace{.15in} S(e_{12}) = -q e_{12}, 
\hspace{.15in} S(e_{21})=-q^{-1} e_{21}, \hspace{.15in} S(e_{22})=e_{11}.$$
We have that $\mc{O}_q(SL_2(k))$ coacts on $k_q[u,v]$ from the 
right by $$\rho(u) = u \otimes e_{11} + v \otimes e_{21} \hspace{.2in} 
\text{and} \hspace{.2in} \rho(v) = u \otimes e_{12} + v \otimes e_{22}.$$
\end{example}

We want to restrict ourselves to $H$-actions that do not factor
through `smaller' Hopf algebras.

\begin{definition}\cite{BanicaBichon}
\label{def1.5} Let $M$ be a left $H$-module. We say that $M$ is an {\it
inner faithful} $H$-module, or $H$ \emph{acts inner faithfully} on $M$,
if $IM\neq 0$ for every nonzero Hopf ideal $I$ of $H$.

Dually, let $N$ be a right $K$-comodule. We say that $N$ is an
\emph{inner faithful} $K$-comodule, or $K$ \emph{coacts inner faithfully}
on $N$, if for any proper Hopf subalgebra $K'\subsetneq K$, $\rho(N)$
is not in $N\otimes K'$.
\end{definition}

Let $U$ be any left $H$-submodule of $R$ that generates $R$ as an
algebra. Then $H$ acts on $R$ inner faithfully if and only if $H$
acts on $U$  inner faithfully. Here is a useful lemma pertaining to
inner faithfulness.

\begin{lemma}
\label{lem1.6}
Suppose $H$ is finite dimensional and let $K =H^{\circ}$. Let $U$ be a
left $H$-module, so it is also a right $K$-comodule with coaction $\rho$.
Then the following conditions hold.
\begin{enumerate}
\item
The $H$-action on $U$ is inner faithful if and only if the induced
$K$-coaction on $U$ is inner faithful.
\item
Let $C$ be the smallest subcoalgebra of $K$ such that
$\rho(U)\subseteq U\otimes C$.  Then $U$ is an inner faithful $K$-comodule
if and only if $K$ is generated as an algebra by $\bigcup_{n\geq 0}
S^n(C)$. \qed
\end{enumerate}
\end{lemma}

We define the homological determinant of an $H$-action on an algebra 
$R$ below. We refer to \cite[Section~2]{KKZ:Gorenstein} for 
background on local cohomology modules.

\begin{definition}
\label{def1.7}
Let $H$ be a finite dimensional Hopf algebra acting on a connected graded
noetherian AS Gorenstein algebra $R$. Suppose that the $H$-action on $R$
preserves the grading of $R$. Let $d=\mathrm{injdim}(R)$ which is finite since
$R$ is AS Gorenstein. We denote by $\fm$ the maximal graded ideal of 
$R^{H}$ consisting of all elements with positive degree, and 
$H^{d}_{\fm}(R)$ the $d$-th local cohomology of $R$ with respect to $\fm$.
The lowest degree nonzero homogeneous component of
$H^d_{\fm}(R)^{\ast}$ is $1$-dimensional, and for it, we
choose a basis element $\mathfrak{e}$. Then there is an algebra
homomorphism $\eta: H \rightarrow k$ such that the right $H$-action on $H_{\mathfrak{m}}^d(R)^*$ is given by $\eta(h)\mathfrak{e}$ for all $h \in H$.

\begin{enumerate}
\item[(1)] The composite map $\eta
\circ S: H \rightarrow k$ is called the {\it homological
determinant} of the $H$-action on $R$, denoted by $\hdet_H R$.
\item[(2)] We
say that $\hdet_H R$ is {\it trivial} if $\eta \circ S = \epsilon$
\cite[Definition 3.3]{KKZ:Gorenstein}.
\end{enumerate}

Dually, if $K$ coacts on $R$ from the right, then $K$ coacts on
$k{\mathfrak e}$ and $\rho({\mathfrak e})={\mathfrak e}\otimes {\sf D}^{-1}$
for some grouplike  element ${\sf D}$ in $K$.
\begin{enumerate}
\item[(3)] The {\it homological
codeterminant} of the $K$-coaction on $R$ is defined to be
${\mathrm{hcodet}}_K R={\sf D}$.
\item[(4)] We say that ${\mathrm{hcodet}}_K R$
is {\it trivial} if ${\mathrm{hcodet}}_K R=1_K$
\cite[Definition 6.2]{KKZ:Gorenstein}.
\end{enumerate}
\end{definition}

Let $V$ be a $k$-vector space and $G$ be a finite subgroup of $GL(V)$
acting linearly on $R=k[V]$. Then $\hdet_{kG} R = \det$ where
$\det:kG\to k$ is the determinant map on $GL(V)$, restricted to $G$, then
extended linearly to $kG$. Consequently, $\hdet_{kG}R$ is trivial if
and only if $G\subseteq SL(V)$, see \cite{JorgensenZhang}.
Moreover in \cite{KKZ:Gorenstein}, the
authors assume that $H$ is semisimple, yet the definition of
homological (co)determinant does not require that $H$ is semisimple.


\section{Hopf actions with trivial homological determinant}
\label{sec2}

In this section, we discuss several results pertaining to Hopf
actions on graded algebras with trivial homological determinant.
In the case of Hopf algebras acting on AS regular algebras of global dimension
$2$, we can express the homological determinant concretely.

\begin{theorem} \label{thm2.1}
Let $R=k\langle U\rangle/(r)$ be an AS regular algebra of
global dimension $2$ generated in degree one and let $H$ be a Hopf algebra that acts on $R$. Then, $\hdet_H R$ is equal to the map $H\to \End_k(M)$
where $M$ is the 1-dimensional $H$-module $(kr)^{\ast}$. Consequently,
$\hdet_H R$ is trivial if and only if $k r$ is the trivial $H$-module.
\end{theorem}

\begin{proof}
First, we construct an $H$-invariant free resolution of the left
$R$-module $_Rk$. Since $k$ is algebraically closed, we may assume
that $r=vu-quv+\eta u^2$ for some $q\in k^\times$ and $\eta \in k$, where
$\{u,v\}$ is a basis of the $k$-vector space $U$ [Example
\ref{ex1.2}]. Since $R$ is an $H$-module algebra, the multiplication map
$\mu: R\otimes R\to R$ is an $H$-homomorphism. This induces an
$H$-homomorphism $\phi: R\otimes U\to R$. By using the chosen basis,
we have an $H$-homomorphism $d_1:(R\otimes ke_u) \oplus (R\otimes
ke_v)\to R$ induced by the multiplication map where $e_u:=u\in U$ and
$e_v:=v\in U$. This gives rise to a partial resolution of $_Rk$,
$$(R\otimes ke_u) \oplus (R\otimes ke_v) \xrightarrow{d_1} R
\xrightarrow{d_0} k\xrightarrow{\quad} 0.$$ Since $R$ has global
dimension 2, the kernel of the map $d_1$ is a free module of rank 1,
denoted by $Re_r$ with basis element $e_r$. Using the relation $r$, we have that
$e_r$ is identified with element $(v + \eta u)\otimes e_u- qu\otimes e_v$
in $(R\otimes ke_u) \oplus (R\otimes ke_v)$. By definition, $kr$ is
a left $H$-module. Hence, $ke_r$ is a left $H$-module which is
isomorphic to $kr$. The $H$-module structure on $Re_r$ is equivalent
to the $H$-module structure on $R\otimes ke_r$. Now we have an
$H$-equivariant free resolution of $_Rk$,
$$0\to R\otimes ke_r\to (R\otimes ke_u) \oplus (R\otimes ke_v)
\to R\to k\to 0.$$ Applying $\Hom_R(-,k)$ to the resolution above, we
obtain an $H$-module isomorphism
$$\Ext^2_R(k,k)\cong\Hom_k(ke_r,k)=: (ke_r)^*.$$
Thus the $H$-action on $\Ext^2_R(k,k)$ is the $H$-action on
$(ke_r)^*$. By \cite[Lemma 5.10(c)]{KKZ:Gorenstein}, the homological
determinant can be computed as the $H$-action on $\Ext^2_R(k,k)$, so
the result follows, and the consequence is clear.
\end{proof}

The following lemma could be useful for future work.

\begin{lemma}
\label{lem2.2}
Let $R=k\langle U\rangle /(r_1,\cdots,r_s)$ be a Koszul AS regular
algebra of global dimension $d$, generated in degree one, and let $H$ be a
Hopf algebra that acts on $R$. If $\hdet_H R$ is trivial, then
$U^{\otimes d}$ contains a copy of the trivial $H$-module $k$.
\end{lemma}

\begin{proof}
Recall that the Koszul dual $R^!$ of $R$ is defined to be
$k\langle U^* \rangle/(r_1, \dots, r_s)^{\perp}$. Since $R$ is AS regular,
$R^!$ is isomorphic to
$$\Ext_R^*(k,k) = k \oplus \Ext_R^1(k,k) \oplus \dots \oplus \Ext_R^d(k,k).$$
Here, $\Ext_R^1(k,k) \cong U^*$. Since $R^!$ is Koszul, it is
generated in degree one. Thus $\Ext_R^d(k,k)$ is isomorphic to a
quotient of $(U^*)^{\otimes d}$ by some $k$-vector space. Since
$\hdet_H R$ is trivial, we have that $\Ext_R^d(k,k) \cong k$ as
$H$-modules by \cite[Lemma 5.10(c)]{KKZ:Gorenstein}. Thus, $k$ is a
quotient module of the $H$-module $(U^*)^{\otimes d}$, which is
equivalent to our desired result.
\end{proof}

Our main application of Hopf actions on algebras with trivial
homological determinant is illustrated in the following theorem.

\begin{theorem} \label{thm2.3}
Let $H$ be a semisimple Hopf algebra. Suppose that $R$ is a noetherian
connected graded AS regular algebra and $H$ acts on $R$. If
$R^H\neq R$ and $\hdet_H R$ is trivial, then $R^H$ is not AS regular.
\end{theorem}

By \cite[Theorem 0.1]{KKZ:Gorenstein}, $R^H$ is always AS Gorenstein
under the hypotheses above.
On the other hand, we set some notation to prove the theorem above.  
The following lemma is proved in \cite{StephensonZhang}.

\begin{lemma} \cite[Sections 2 and 3]{StephensonZhang}
\label{lem2.4} Let $A$ be a noetherian AS regular algebra. The
Hilbert series $H_A(t)$ is equal to $\frac{1}{p(t)}$, where
$p\in \mathbb{Z}[t]$ satisfies
\begin{enumerate}
\item
$p(0)=1$,
\item
\cite[Proposition 3.1(4)]{StephensonZhang} $\deg p(t)=\ell$, where
$\ell$ is the AS index of $A$,
\item
\cite[Theorem 2.4(2)]{StephensonZhang}
the leading coefficient of $p(t)$ is $1$ or $-1$. \qed
\end{enumerate}
\end{lemma}

When we are working with more than one algebra, let us use $p_A(t)$
to denote the polynomial $p(t)$. For a graded module $M$, let  $M(n)$
denote the shift of $M$ by degree
$n$. The following lemma is well-known
and follows from \cite[Lemma 1.10]{KKZ:Rigidity}.

\begin{lemma}
\label{lem2.5}
Let $A$ be a noetherian connected graded algebra,
and let $B$ be a noetherian connected graded subalgebra of $A$.
Suppose that
\begin{enumerate}
\item[(i)]
$B\neq A$;
\item[(ii)]
$A$ and $B$ are AS regular algebras of the same global dimension; and
\item[(iii)]
$_BA$ and $A_B$ are finitely generated left and right $B$-modules
respectively.
\end{enumerate}
Then the following hold.
\begin{enumerate}
\item
Both $_BA$ and $A_B$ are graded free $B$-modules.
\item
Suppose that $H_A(t)=\frac{1}{p_A(t)}$ and $H_B(t)=\frac{1}{p_B(t)}$.
Then, $p_A(t)\mid p_B(t)$, and $\deg p_A(t)<\deg p_B(t)$.
\end{enumerate}
\end{lemma}

\begin{proof} (a) This follows from \cite[Lemma 1.10(b)]{KKZ:Rigidity}.

(b) By part (a), $A$ is a finitely generated free graded left
$B$-module. Hence, we can write $_BA= \bigoplus_{n=0}^w B(n)^{i_n}$
for a finite sequence of non-negative integers $i_0,i_1,\cdots,i_w$.
Then
$$\frac{1}{p_A(t)}=H_A(t)=\left(\sum_{n=0}^w i_n t^n\right)H_B(t)=
\frac{\sum_{n=0}^w i_n t^n}{p_B(t)},$$ which implies that
$p_B(t)=p_A(t)\left(\sum_{n=0}^w i_n t^n\right)$ and $p_A(t)\mid
p_B(t)$. By Lemma \ref{lem2.4}(a), $p_A(0)=p_B(0)=1$, so we have that
$i_0=1$. Since $B\neq A$, we get that $\sum_{n=0}^w i_n t^n\neq 1$. Thus $\deg
(\sum_{n=0}^w i_n t^n)>0$. Therefore $\deg p_A(t)<\deg p_B(t)$.
\end{proof}

The following lemma follows from \cite[Section 3]{KKZ:Gorenstein}.

\begin{lemma}
\label{lem2.6}
Let $A$ be a noetherian AS Gorenstein algebra, and let
$H$ be a semisimple Hopf algebra acting on $A$. Suppose $B:=A^H$ is
AS Gorenstein. Then, we have the following statements:
\begin{enumerate}
\item
$\ell_A\leq \ell_B$, where $\ell_A$ and $\ell_B$ denote the
AS indexes of $A$ and $B$, respectively; and
\item
$\hdet_H A$ is trivial if and only
if $\ell_A=\ell_B$.
\item
Let $$k\to H'\to H \to H''\to k$$ be an exact sequence of
semisimple Hopf algebras. Let $C=A^{H'}$ and let $H''$ act 
on $C$ naturally. If $\hdet_H A$ is trivial, then so is
$\hdet_{H''} C$.
\end{enumerate}
\end{lemma}

\begin{proof} Let $\ell:=\ell_A$ be the AS index of $A$. We combine
the proof of (a) and of (b) as follows.

We  use the convention introduced in
\cite[Section 3]{KKZ:Gorenstein}, in particular, we assume that $H$
acts on $A$ on the right. Let $Y$ denote the local cohomology
module $H_{\fm}^d(A)^*$, where $\mathfrak m=A_{\ge 1}$ and
$d=\injdim A$. By \cite[Lemma 1.7]{KKZ:Gorenstein}, $Y$ is
isomorphic to $H^d_{\fm'}(A)^{\ast}$, where $\fm'$ is the maximal
graded ideal of $A^H$, since $A$ is finite
over $A^H$ as left and right modules. We also denote by $\Omega$, the
$A^H$-module $Y\cdot \int$ where $\int$ is an integral in $H$.
Then by \cite[Lemma~3.2(b)]{KKZ:Gorenstein}, $\Omega[d]$ is the
balanced dualizing complex of $A^H$. If $\hdet_H A$ is trivial, then by
\cite[Lemma 3.5(c,d)]{KKZ:Gorenstein}, $Y^H\cong \Omega \cong A^H(-\ell)$
as right $A^H$-modules. Therefore, by  \cite[Lemma 1.6]{KKZ:Gorenstein},
the AS index of $A^H$ is equal to $\ell$.  This proves one
implication of part (b).

Conversely, assume that $\hdet_H A$ is not trivial. By
\cite[Lemma 3.5(f)]{KKZ:Gorenstein}, ${\mathfrak e}\not\in Y^H$, where
${\mathfrak e}$ is a generator of $Y$. Thus, the lowest degree of
any nonzero elements in $Y^H$ is larger than $\ell$. Since $Y^H[d]$
($=\Omega[d]$) is the balanced dualizing complex over $B$,  we have 
that $Y^H[d]\cong {^\sigma B^1}(-\ell_B)[d]$ by 
\cite[Lemmas~3.2(b) and~1.6]{KKZ:Gorenstein}. Hence $Y^H\cong B(-\ell_B)$ 
as right $B$-modules. Thus the lowest degree of any nonzero elements 
in $B(-\ell_B)$ is larger than $\ell$. So $\ell_B>\ell$. This proves 
the other implication of (b), as well as part (a).

(c) Since $\hdet_H A$ is trivial, we have that $B:=A^H$ is AS Gorenstein.
Since $H'$ is a Hopf subalgebra of $H$, $\hdet_{H'}A$ is trivial.
Consequently, $C$ is AS Gorenstein.
Clearly, $B=A^H=C^{H''}$. By part (a), $\ell_A\leq \ell_C\leq \ell_B$.
By part (b), $\ell_A=\ell_B$, and hence $\ell_C=\ell_B$. The assertion follows
from part (b).
\end{proof}

Now, we are ready to prove Theorem \ref{thm2.3}.

\begin{proof}[Proof of Theorem \ref{thm2.3}] Proceed by
contradiction and assume that $B:=R^H$ is AS regular. By the hypotheses
of Theorem \ref{thm2.3}, $R$ is noetherian and AS regular. The
hypothesis in Lemma \ref{lem2.5}(iii) holds because $H$ is a semisimple 
Hopf algebra \cite[Theorem 4.4.2]{Montgomery}. Since $R$ and $B$ have the 
same injective dimension, they have the same global dimension when both
have finite global dimension. By Lemma \ref{lem2.5}(b), $\deg
p_R(t)<\deg p_B(t)$. Hence,
$\ell_R=\deg p_R(t)<\deg p_B(t)=\ell_B$ due to Lemma \ref{lem2.4}(b).
By Lemma \ref{lem2.6}(b), $\hdet_H R$ is not trivial, thus producing
a contradiction.
\end{proof}

Next, we recall results from \cite{CWZ:Nakayama} on finite dimensional 
Hopf actions on AS regular algebras of global dimension 2, sometimes 
with trivial homological determinant. These results are used later,
especially in Section \ref{sec6}.

\begin{proposition}
\label{pro2.7}
Let $H$ be a finite dimensional Hopf algebra. Let $R$ be an AS regular
algebra of global dimension 2, that is to say, $k_q[u,v]$ or $k_J[u,v]$. Suppose that
$H$ acts on $R$ inner faithfully and preserves the grading of $R$.
\begin{enumerate}
\item
\cite[Theorem 5.10]{CWZ:Nakayama}
If $R$ is either $k_J[u,v]$ or $k_q[u,v]$ for $q$ not a root of unity, then
$H$ is a group algebra.
\item
\cite[Theorem 0.1 and Remark 4.4]{CWZ:Nakayama}
If $R=k_q[u,v]$ with $q = \pm 1$ and if the $H$-action on $R$ has
trivial homological determinant, then $H$ is semisimple.
\item
Assume Hypothesis \ref{hyp0.2} with $H$ non-semisimple.
Then $R = k_{q}[u,v]$ for $q$ a root of unity with $q \neq \pm 1$.
\end{enumerate}
\end{proposition}

\begin{proof} Part (c) follows from parts (a,b).
\end{proof}

The following result provides the classification of finite dimensional, 
cocommutative Hopf algebras $H$ acting on $R$ satisfying Hypothesis 
\ref{hyp0.2}.  Recall that $H$ is a finite group algebra in this case. 
Moreover, if $R = k[u,v]$, then $G$ is a finite subgroup of $SL_2(k)$, 
so we cover the cases where $R$ is noncommutative below.

\begin{proposition}
\label{pro2.8}
Let $R$ be a noncommutative AS regular algebra of dimension $2$ and $kG$
be a finite group algebra acting on $R$ with trivial homological determinant. 
Then, assuming that $G\neq 1$, we have the following statements.
\begin{enumerate}
\item
If $R \simeq k_J [ u, v ]$, then $G = C_2 = \langle
\sigma \mid \sigma^2 \rangle$ with $\sigma ( z ) = - z$ for
all $z \in R_1$.
\item
If $R \simeq k_q [ u, v ]$ and $q\neq \pm 1$,
then $G = C_n = \langle
\sigma \mid \sigma^n \rangle$ with $\sigma ( u ) = \zeta u$,
$\sigma ( v ) = \zeta^{- 1} v$, where $\zeta$ is a primitive
$n$-th root of unity.
\item
If $R \simeq k_{- 1} [ u, v ]$, then either
\begin{enumerate}
\item[$\ast$] $G=C_n$ with action as in part (b),
\item[$\ast$] $G = C_2 = \langle \sigma \mid \sigma^2 \rangle$ with
$\sigma ( u ) = v$ and $\sigma ( v ) = u$, or
\item[$\ast$] $G = D_{2 n} = \langle \sigma, \tau \mid \sigma^n, 
\tau^2, \sigma \tau \sigma \tau \rangle$ with $\sigma (u) = \zeta u$,
$\sigma ( v ) = \zeta^{- 1} v$ and $\tau (u) = v$, $\tau ( v ) = u$,
where $\zeta$ is a primitive $n$-th root of unity.
\end{enumerate}
\end{enumerate}
\end{proposition}

\begin{proof}
Suppose that  $\sigma \in G$ acts on $R$ by $\sigma ( u ) = 
a_{11} u + a_{21} v$ and $\sigma ( v ) = a_{12} u + a_{22} v$ for 
some $a_{i j} \in k$.  Since $k$ is algebraically closed, we can 
assume that the relation $r$ of $R$ is of the form $vu-quv-\eta u^2$ 
for $q \in k^{\times}$ and $\eta = 0,1$. Moreover,
\[
\begin{array}{rl}
    \sigma ( r )  = &
(  ( 1 - q )a_{11} a_{12} - \eta  a_{11}^2 ) u^2
+ ( a_{12} a_{21} - q a_{11} a_{22} - \eta a_{11} a_{21} )
    u v\\
&+ ( a_{11} a_{22} - q a_{12} a_{21} - \eta a_{11} a_{21} )
    v u  + (  ( 1 - q ) a_{21} a_{22} - \eta  a_{21}^2 ) v^2.
\end{array}
\]
By Theorem \ref{thm2.1}, the trivial homological condition implies that
$\sigma(r)=r$. Hence we have the equations
$$\begin{array}{rlcrl}
       ( 1 - q ) a_{11} a_{12} - \eta a_{11}^2  &= -\eta &
       \hspace{1em} & ( 1 - q )  a_{21} a_{22} - \eta a_{21}^2  &=
       0\\
       a_{11} a_{22} - q a_{12} a_{21} - \eta a_{11} a_{21}  &= 1 &  &
       a_{12} a_{21} - q a_{11} a_{22} - \eta a_{11} a_{21}  &= - q.
\end{array}$$

(a) Suppose that $q = \eta = 1$. The top two equations yield 
$a_{11}^2 = 1$ and $a_{21} = 0$. The bottom two equations then yield 
$a_{11} a_{22} = 1$, so $a_{11} = a_{22} = \pm 1$. Now 
$\sigma^m ( v ) = m a_{11}^{m - 1} a_{12} u + a^m_{11} v$. Since $\sigma$ 
has finite order and $a_{11} \neq 0$, we see that $a_{12} = 0$. This
proves the claim in part (a).

(b) Suppose that $\eta = 0$ and $q\neq \pm 1$. Since $q \neq 1$, the 
top two equations yield $a_{11} a_{12} = 0$ and $a_{21} a_{22} = 0$. 
Since $1- q^2 \neq 0$, the bottom two equations also give 
$a_{11} a_{22} = 1$ and $a_{12} a_{21} = 0$, so $a_{12} = a_{21} = 0$.
Since $G$ has finite order, it is cyclic with the action on $R$ as 
shown in  case (b).

(c) Lastly, suppose that $\eta = 0$ and $q = - 1$. First,
observe that if the equations $a_{11} a_{22} = 1$ and $a_{12} = a_{21} = 0$
still hold, then $G=C_n$ as in part (b).

The bottom two equations both become $a_{11} a_{22} + a_{12} a_{21} = 1$, 
so we have more solutions. From the top two equations, one of the 
$a_{i j}$ is zero. If $a_{12} = 0$ (resp. $a_{21} = 0$) then 
$a_{11} a_{22} = 1$, so $a_{21} = 0$ (resp. $a_{12} = 0$). So, for $\zeta$ 
a primitive $n$-th root of unity, we have that $\sigma_{\zeta}$ is an 
automorphism of $R$ in this case, where $\sigma_{\zeta}(u) = \zeta u$ 
and $\sigma_{\zeta}(v) = \zeta^{-1} v$.

If $a_{11} = 0$ or $a_{22} = 0$, then $a_{12} a_{21} =1$. So, 
$a_{12} = \lambda$, but up to base change, $a_{12} = a_{21} =1$. Thus 
we have another automorphism $\tau$ of $R = k_{-1}[u,v]$ with 
$\tau(u) = v $ and $\tau(v) = u$.  Now suppose that $\tau_{\alpha}$ 
with $\tau_{\alpha}(u) = \alpha v$ and $\tau_{\alpha}(v) = \alpha^{-1}u$ 
is in $G$ for $\alpha \in k^{\times}$. Then $\tau_{\alpha} \tau 
=: \sigma_{\alpha}$, which implies that $\alpha$ must be a root of 
unity. This shows that $G = \langle \tau \rangle$ or
$G = \langle \sigma_{\zeta}, \tau \rangle$. This proves the claim in (c).
\end{proof}

\begin{remark}
\label{rem2.9}
It is worth noting that $SL_{2}(k)$ does not contain any finite subgroup
isomorphic to a dihedral group $D_{2n}$. In other words, 
there is no inner faithful action of $D_{2n}$ on $k[u,v]$ having trivial 
homological determinant. In contrast, by part (c) of the above 
proposition, there does exist such an action on the (noncommutative) skew
polynomial ring $k_{-1}[u,v]$. This phenomenon should be further 
investigated. See Remark \ref{rem7.2}, for instance.
\end{remark}


\section{Self-dual modules}
\label{sec3}

We study self-dual modules of Hopf algebras in this section, objects
which are relevant to the material in the previous section (see
Corollary \ref{cor3.3}), and play a crucial role in the proof of
part (a) of Theorem \ref{thm0.4} in Section \ref{sec4}.

Given a left $A$-module $M$, there is a natural right $A$-module
structure on $M':=\Hom_k(M,k)$, which is defined by
$$ (f\bullet a)(m)=f(am)$$
for all $f\in M'$, $a\in A$ and $m\in M$. 
If $H$ is a Hopf algebra with
antipode $S$, then for every finite dimensional left $H$-module $M$, we
define a left $H$-module structure on $M'$ by
\begin{equation}
\label{E3.0.1}\tag{E3.0.1}
a \ast f= f\bullet S(a)
\end{equation}
for all $a\in H$ and $f\in M'$. Then $M'$ becomes a left $H$-module,
which is now denoted by $M^*$. 

\begin{definition}
\label{def3.1} Let $U$ denote a  left $H$-module. We say $U$ is {\it
self-dual} if $U\cong U^*$ as left $H$-modules.
\end{definition}

We have the following basic results.

\begin{lemma} \label{lem3.2}
Let $U$ and $V$ be simple left $H$-modules. Then the following statements hold.
\begin{enumerate}
\item
$U\cong V^*$ if and only if the trivial $H$-module $k$ is a quotient
module of $U\otimes V$.
\item
If $U\cong V^*$, then the quotient map $U\otimes V\to k$ is unique up to
a scalar.
\item
$U$ is self-dual if and only if $k$ is a quotient $H$-module of
$U^{\otimes 2}$ and if and only if $k$ is an $H$-submodule of
$U^{\otimes 2}$.
\end{enumerate}
\end{lemma}

\begin{proof} (a)
By $\Hom-\otimes$ adjointness \cite[Lemma 1.1]{BrownGoodearl},
there is an isomorphism
$$\Hom_H(U\otimes V, k)\cong \Hom_H(U,\Hom_k(V,k))=\Hom_H(U,V^*).$$
The rightmost term is nonzero if and only if $U\cong V^*$, as $U$
and $V$ are simple $H$-modules. The leftmost term is nonzero if and
only if $k$ is a quotient module of $U\otimes V$. The assertion
follows.

(b) The assertion follows from the fact that $\dim_k \Hom_H(U,V^*) =
\dim_k \Hom_H(U,U)=1$.

(c) The first assertion follows from part (a) by setting $V=U$. 
It is easy to see that $U$ is self-dual if and only if
$U^*$ is self-dual. The second
assertion follows by applying $(-)^*$.
\end{proof}

By \cite[Corollary 6]{KSZ}, if $H$ is semisimple and has a
nontrivial self-dual simple module, then the dimension of $H$ is
even.

\begin{corollary} \label{cor3.3}
Let $H$ be a finite-dimensional Hopf algebra acting inner faithfully
on an AS regular algebra $R=k\langle U \rangle/(r)$ of global
dimension 2. Suppose that $U$ is a simple $H$-module and that 
$\hdet_H R$ is trivial. Then
\begin{enumerate}
\item $k$ is an $H$-submodule of $U^{\otimes 2}$,
\item $k$ is a quotient $H$-module of $U^{\otimes 2}$, and
\item $U$ is a self-dual $H$-module.
\end{enumerate}
\end{corollary}

\begin{proof}
Applying Theorem \ref{thm2.1}, one sees that $U^{\otimes 2}$
contains $k$ as an $H$-submodule. The rest follows from Lemma
\ref{lem3.2}.
\end{proof}

Another application of self-duality pertains to the antipode $S$ of
$H$.

\begin{lemma} 
\label{lem3.4}
Let $H$ is a finite dimensional Hopf algebra. Suppose that $H$ has
an algebra decomposition $H=C\oplus I$ where $C$ and $I$ are ideals
of $H$ so that $C\cong M_n(k)$. Then, $S(C)\subset C$ if and only if
$C$ is self-dual as a left $H$-module (viewing $C$ as $H/I$).
\end{lemma}

\begin{proof}
Viewing $C$ as a left $H$-module, we have that $IC=0$. Then
$C':=\Hom_k(C,k)$ has a right $H$-module structure and $C' I=0$. The
dual $H$-module $C^*$ is a left $H$-module defined by
\eqref{E3.0.1}. Thus $S^{-1}(I)C^*=0$. Since $S$ is an
anti-automorphism of $H$,  we have that $S^{-1}(C)$ and 
$S^{-1}(I)$ are ideals of $H$ such that $H=S^{-1}(C)\oplus S^{-1}(I)$.

If $S(C)\subset C$ (or equivalently, $S^{-1}(C)=C=S(C)$ since $C$ is
finite dimensional and $S$ is bijective), then $S^{-1}(I)=I$ by the
fact that $I$ is uniquely determined by $C$. Hence $IC^*=0$. Since
$\dim C^*=\dim C$ and $C$ is a simple $k$-algebra, $C^*\cong C$ as
left $H$-modules. So, $C$ is self-dual.

Conversely, if $C$ is self-dual, then $IC^*=IC=0$, which implies
that $C' S(I)=0$. Since $I=r.\mathrm{ann}_H(C')$, $S(I)\subset I$. As a
consequence $S(I)=I$, which implies that $S(C)=C$.
\end{proof}

\noindent Note that Lemma \ref{lem3.4} can be applied when $H$ is semisimple,
in which case $H$ is a direct sum of simple rings.

Now, we turn our
attention to nondegenerate invariant bilinear forms. Let $U$ be a
simple left $H$-module. By $\Hom-\otimes$ adjointness, there is a
nondegenerate invariant bilinear form
$$\langle \cdot,\cdot\rangle: U^*\times U\to k \hspace{.1in}
\text{satisfying} \hspace{.1in}
\langle h_1\cdot u, h_2\cdot v\rangle =\epsilon(h) \langle
u,v\rangle$$
for all $h\in H$ and all $u\in U^*$ and $v\in U$. Here,
$\Delta(h) = \sum h_1 \otimes h_2$. Since $k$ is
algebraically closed,
$$\dim_k \Hom_H(U^*\otimes U,k)=\dim_k \Hom_H(U^*,U^*)=1.$$
This implies that the nondegenerate invariant bilinear form on
$U^*\times U$ is unique up to scalar multiplication.
Thus, if $U$ is self-dual, there is a nondegenerate invariant bilinear
form on $U$
$$\langle \cdot,\cdot\rangle: U\times U\to k \hspace{.1in}
\text{satisfying} \hspace{.1in}
\langle h_1\cdot u, h_2\cdot v\rangle =\epsilon(h) \langle
u,v\rangle$$
for all $h\in H$ and all $u,v\in U$.

A self-dual simple $H$-module admits a nondegenerate invariant
bilinear form which is either symmetric or skew-symmetric
\cite[Theorem 3]{KSZ}. Therefore, we proved the following lemma.

\begin{lemma} \label{lem3.5} \cite[Theorem 3]{KSZ}
Let $H$ be a semisimple Hopf algebra and $U$ be a self-dual 
simple left $H$-module.
Then there is a nondegenerate invariant bilinear form, unique up to
scalar multiplication,
$$\langle \cdot,\cdot\rangle: U\times U\to k.$$
Furthermore, such a nondegenerate invariant bilinear form
is either symmetric or skew-symmetric. \qed
\end{lemma}

In particular, we use this lemma to study self-dual left
$H$-modules of dimension 2.

\begin{lemma} \label{lem3.6}
Let $H$ be a semisimple Hopf algebra and $U$ be a 2-dimensional 
self-dual simple
left $H$-module. Then, there is a basis of $U$, say $\{u, v\}$, such
that the trivial module $k$, as a direct summand in $U^{\otimes 2}$,
has a basis element of the form $u^2+v^2$ or $uv-vu$.
\end{lemma}

\begin{proof}
Let $V = U^{\ast}$. Then $V$ is a $2$-dimensional self-dual
$H$-module. By Lemma \ref{lem3.5}, we can choose a bilinear form
$\langle \cdot, \cdot \rangle : V \times V \longrightarrow k$ which
induces a morphism $\vartheta : V \otimes V \longrightarrow k$ of
$H$-modules where $k$ is the trivial $H$-module. Consider the dual
morphism $\vartheta^{\ast} : k \longrightarrow \left( V \otimes V
\right)^{\ast} \simeq V^{\ast} \otimes V^{\ast} \simeq U^{\otimes
2}$. This map is split since $H$ is semisimple, that is,
$\vartheta^{\ast} \left( 1 \right)$ generates a trivial direct
summand of $U^{\otimes 2}$.

We now show that we can choose a basis such that $\vartheta^{\ast}
\left( 1 \right)$ has the desired form. By Lemma \ref{lem3.5},
$\langle \cdot, \cdot \rangle$ is either symmetric or
skew-symmetric, so by linear algebra, we can choose a basis $\left\{
v_1, v_2 \right\}$ of $V$ such that $\vartheta \left( v_i \otimes
v_j \right) = \delta_{i j}$ (if $\langle \cdot, \cdot \rangle$ is
symmetric) or $\vartheta \left( v_i \otimes v_j \right) = \left( - 1
\right)^i \left( \delta_{i j} - 1 \right)$ (if $\langle \cdot, \cdot
\rangle$ is skew-symmetric). Therefore $\vartheta^{\ast} \left( 1
\right) = v^{\ast}_1 \otimes v^{\ast}_1 + v^{\ast}_2 \otimes
v^{\ast}_2$ or $\vartheta^{\ast} \left( 1 \right) = v^{\ast}_1
\otimes v^{\ast}_2 - v_2^{\ast} \otimes v_1^{\ast}$ as desired.
\end{proof}


\section{Proof of Theorem \ref{thm0.4}: $H$ is noncommutative and semisimple}
\label{sec4}

The goal of this section is to prove Theorem \ref{thm0.4} in the case 
where $H$ is a semisimple Hopf algebra and the $H$-module $U$ of 
$R=k \langle U \rangle / (r)$ is simple. The first step is to show that $H$ is
noncommutative precisely when  $U$ is not a direct sum of two 
1-dimensional simple left $H$-modules.

\begin{lemma} \label{lem4.1}
Let $H$ be a finite dimensional Hopf algebra that acts on $R$ as above. Then, we have the following statements.
\begin{enumerate}
\item
The Hopf algebra $H$ is commutative if and only if the 2-dimensional
left $H$-module $U$ is a direct sum of two 1-dimensional simple $H$-modules.
\item
We have that $H$ is noncommutative if and only if $U$ is not a direct sum of two
1-dimensional simple $H$-modules.
\item
Suppose that $H$ is semisimple. Then, $H$ is noncommutative if and only if
$U$ is simple.
\end{enumerate}
\end{lemma}

\begin{proof} (a) If $H$ is commutative, then $H$ is semisimple and all of
its simple modules are 1-dimensional. Therefore $U$ is a direct sum
of two 1-dimensional simple $H$-modules. Conversely, assume that $H$ acts on a
direct sum of 1-dimensional left $H$-modules $U = T_1 \oplus T_2$.
Then, each element in $H$ acts on $T_i$ as a scalar multiplication,
and hence the Hopf ideal $[H,H]$ acts as zero on $T_i$. Thus, the
$H$-action on $U$ factors through $[H,H]$, and $H$ is commutative.

(b) This is equivalent to (a).

(c) Suppose $H$ is semisimple. Then $U$ must be simple as it is not
a direct sum of two 1-dimensional simple $H$-modules. Then, part (c) is
equivalent to part (b).
\end{proof}

Now we specify the AS regular algebras $R$ of global dimension two that
occur in Theorem \ref{thm0.4}, when $H$ is noncommutative and semisimple.

\begin{proposition}
\label{pro4.2}
Assume Hypothesis \ref{hyp0.3}. If $H$ is noncommutative, then the
algebra $R$ is isomorphic to either $k[u,v]$ or $k_{-1}[u,v]$.
\end{proposition}

\begin{proof}
Note that $R$ is an AS regular algebra of global dimension two,
generated in degree one, of the form $k\langle U\rangle /(r)$ for some
$r\in U^{\otimes 2}$. By Lemma \ref{lem4.1}(c), $U$ is simple. By
Corollary \ref{cor3.3}, $U$ is self-dual. Also, by 
Lemma~\ref{lem3.6}, $r$ is either $u^2+v^2$ (which is equivalent to
$r=uv+uv$ after a base change) or $uv-vu$, so the assertion follows.
\end{proof}

We require the following results pertaining to an anti-homomorphism
of a matrix coalgebra. Let $C = M_n \left( k \right)$ be a matrix
coalgebra over an algebraically closed field $k$. Let $\tau : C
\longrightarrow C$ denote the transpose, and $c_N : C
\longrightarrow C$ denote conjugation by $N \in G L_n \left( k
\right)$. Let $S$ and $T$ be $k$-linear automorphisms of $C$. We say
that $S$ is {\it equivalent} to $T$ if $  S  =  c_N^{- 1} \circ T
\circ c_N$ for some $N \in G L_n \left( k \right)$. Note that $\tau$
and $c_M$ obey the following commutation law
\begin{equation}
\tau \circ c_{M^{- 1}} = c_{M^T} \circ \tau .
\label{E4.2.1}\tag{E4.2.1}
\end{equation}
We establish the following lemma.

\begin{lemma}
\label{lem4.3}
Let $M \in G L_n \left( k \right)$, then $c_M \circ \tau$ is equivalent to
$c_{P^T M P} \circ \tau$ for any $P \in G L_n \left( k \right)$.
\end{lemma}

\begin{proof}
It suffices to show that $c_{P^{- 1}}^{- 1} \circ c_M \circ \tau
\circ c_{P^{- 1}} = c_{P^T M P} \circ \tau$. Using (\ref{E4.2.1}) we
have that
$$c_{P^{- 1}}^{-1} \circ c_M \circ \tau \circ c_{P^{- 1}} ~= ~ c_P \circ
c_M \circ c_{P^T} \circ \tau ~= ~~ c_{P^T M P} \circ \tau.$$
\vspace{-.6in}
\end{proof}

\vspace{.5in}

\begin{proposition} \label{pro4.4}
Let $C = M_2 \left( k \right)$ and $S : C \longrightarrow C$ be an
anti-automorphism of order $2$. Then $S$ is equivalent to $c_W \circ
\tau$ where
\begin{eqnarray*}
    W & = & \text{\footnotesize{$ \left(\begin{array}{cc}
      0 & 1\\
      1 & 0
    \end{array}\right)$}} \hspace{.1in} \text{or} \hspace{.1in}
\text{\footnotesize{$ \left(\begin{array}{cc}
      0 & 1\\
      - 1 & 0
    \end{array}\right)$}}.
  \end{eqnarray*}
\end{proposition}

\begin{proof}
Since $S : C \longrightarrow C$ is an anti-automorphism, the composition
$\tau \circ S : C \longrightarrow C$ is an algebra automorphism. By the
Skolem-Noether theorem,
the automorphism $\tau \circ S$ is inner, so we can write
$\tau \circ S =c_{W^{- 1}}$ for some $W \in GL_2 (k)$.
Using \eqref{E4.2.1}, we obtain $S = \tau \circ c_{W^{- 1}}
= c_{W^T} \circ \tau$. Since $S$ has order $2$, we obtain, again
using \eqref{E4.2.1}, that
$$  1  ~ = ~  S^2
           ~ = ~ c_{W^T} \circ \tau \circ c_{W^T} \circ \tau
           ~ = ~ c_{W^T} \circ c_{W^{- 1}} \circ \tau \circ \tau
           ~ = ~ c_{W^{- 1} W^T}.$$
In other words, $W^{- 1} W^T \in Z \left( C \right)$, so
$W^{- 1} W^T$ is a
scalar matrix. Furthermore $\det \left( W^{- 1} W^T \right) = 1$,
from which we conclude that $W^{- 1} W^T = \pm I_{2\times 2}$.
Finally the assertion follows by applying Lemma \ref{lem4.3}.
\end{proof}

Now, we prove Theorem \ref{thm0.4} in the case that $H$ is 
noncommutative and semisimple.

\begin{theorem}
\label{thm4.5} 
Let $H$ be a noncommutative, semisimple Hopf algebra acting on an
algebra $R$ satisfying Hypothesis \ref{hyp0.3}. Then, 
$R=k_{\pm1}[u,v]$ and $U=ku\oplus kv$ is a simple $H$-module. Moreover, 
we have the following statements.
\begin{enumerate}
\item[(a1)]
If $R=k[u,v]$ and $H$ is cocommutative,
then $H\cong k\tilde{\Gamma}$ where $\tilde{\Gamma}$ is a non-abelian
binary polyhedral group.
\item[(a2)]
If $R=k_{-1}[u,v]$ and $H$ is cocommutative, then $H\cong kD_{2n}$
for $n\geq 3$.
\item[(a3)]
If $H$ is noncocommutative, then $R=k_{-1}[u,v]$ and $H^{\circ}$ is a
finite dimensional Hopf quotient of $\cal{O}_{-1}(SL_2(k))$.
\end{enumerate}
\end{theorem}

\begin{proof}
The first assertion follows from Lemma~\ref{lem4.1}(c) and Proposition~\ref{pro4.2}. The 2-dimensional left $H$-module $U$ is simple due to Lemma
\ref{lem4.1}. To proceed, we classify the $K:=H^{\circ}$-coactions on $R$
which induce $H$-actions satisfying our hypotheses.

Since the homological determinant of the $H$-action on $R$ is
trivial,  we have by Corollary \ref{cor3.3} that $U$ is self-dual.
Let $C$ be the smallest sub-coalgebra of $K$ so that $\rho(U)
\subseteq U \otimes C$. Note that $C\cong M_2(k)$ as coalgebras, and $U$
is a simple right $C$-comodule. In particular, $C \cong  U^{\oplus
2}$ as right $C$-comodules. Since $U$ is self-dual, $C$ is also
self-dual. Now, the antipode $S$ of $H$ satisfies $S(C) \subseteq C$
by Lemma \ref{lem3.4}. Let $\{e_{ij}\}_{i,j=1,2}$ be a $k$-vector
space basis for $C$. Then by Lemma \ref{lem1.6}(b), $K=k\langle e_{ij}\mid
i,j=1,2\rangle$. Next, we determine the possibilities for $S$ and the relations
between the $e_{ij}$. This will completely determine $K$ as a Hopf algebra.

Since $H$ is semisimple and char($k$) =0, we have that $S^2=1$, 
so we can apply a version of
Proposition \ref{pro4.4} to the coalgebra $C$.
Hence we can assume that $S=c_W\circ\tau$, as an anti-automorphism
of the coalgebra $C$, where
$W=${\footnotesize $\left(\begin{array}{cc} 0 & 1 \\
p & 0
\end{array} \right)$} for $p^2 =1$.
Let $E$ denote the
coalgebra matrix units
{\footnotesize $\left(\begin{array}{cc} e_{11} & e_{12} \\
e_{21} & e_{22} \end{array} \right)$}. Note that
\begin{equation}
\label{E4.5.1}\tag{E4.5.1}
\Delta(e_{ij}) = \sum_{m=1}^2 e_{im} \otimes e_{mj}.
\end{equation}

By the antipode axiom and \eqref{E4.5.1}, we get that
$$\sum_{m=1}^2 S(e_{im}) e_{mj} ~=~ \sum_{m=1}^2 e_{im} S(e_{mj}) ~=~
\epsilon(e_{ij}) ~=~ \delta_{ij}. $$
Therefore,
$$S(E)E ~=~ (W E^T W^{-1}) E ~=~ I ~=~ E  (W E^T W^{-1}) ~=~ ES(E),$$
which is equivalent to:
{\footnotesize
$$\left(\begin{array}{cc} e_{22} & p^{-1}e_{12} \\ pe_{21} & e_{11}
\end{array} \right)
\left(\begin{array}{cc} e_{11} & e_{12} \\ e_{21} & e_{22}
\end{array} \right) ~=~
\left(\begin{array}{cc} 1 & 0 \\ 0 & 1 \end{array} \right) ~=~
\left(\begin{array}{cc} e_{11} & e_{12} \\ e_{21} & e_{22}
\end{array} \right)
\left(\begin{array}{cc} e_{22} & p^{-1}e_{12} \\ pe_{21} & e_{11}
\end{array} \right).$$}
Hence, $\{e_{ij}\}$ satisfies the following equations:
\begin{align}
e_{22} e_{11} + p^{-1} e_{12} e_{21} &\; =\; 1, \label{E4.5.2}\tag{E4.5.2} \\
e_{22} e_{12} + p^{-1} e_{12} e_{22} &\; =\; 0, \label{E4.5.3}\tag{E4.5.3}\\
e_{11} e_{21} + p e_{21} e_{11} &\; =\; 0, \label{E4.5.4}\tag{E4.5.4}\\
e_{11} e_{22} + p e_{21} e_{12} &\; =\; 1, \label{E4.5.5}\tag{E4.5.5}\\
e_{11} e_{22} + p e_{12} e_{21} &\; =\; 1, \label{E4.5.6}\tag{E4.5.6}\\
e_{12} e_{11} + p^{-1} e_{11} e_{12} &\; =\; 0, \label{E4.5.7}\tag{E4.5.7}\\
e_{21} e_{22} + p e_{22} e_{21} &\; =\; 0, \label{E4.5.8}\tag{E4.5.8}\\
e_{22} e_{11} + p^{-1} e_{21} e_{12} &\; =\; 1. \label{E4.5.9}\tag{E4.5.9}
\end{align}
Equations \eqref{E4.5.5} and \eqref{E4.5.6}, along with Equations
\eqref{E4.5.6} and \eqref{E4.5.9} yield the following two equations:
\begin{align}
e_{12} e_{21} &\;=\; e_{21} e_{12}, \label{E4.5.10}\tag{E4.5.10}\\
e_{11} e_{22} - e_{22} e_{11} &\; =\; (p^{-1}-p)e_{12} e_{21}.
\label{E4.5.11}\tag{E4.5.11}
\end{align}
Observe that Equations \eqref{E4.5.2}$-$\eqref{E4.5.4}, \eqref{E4.5.7},
\eqref{E4.5.8}, \eqref{E4.5.10}, \eqref{E4.5.11} yield precisely 
the relations of the quantum group $\mathcal{O}_{-p^{-1}}(SL_2(k))$ 
from Example~\ref{ex1.4}. Therefore, for $p^2=1$,
$$\mathcal{O}_{-p^{-1}}(SL_2(k)) \twoheadrightarrow K=k\langle C \rangle.$$
Thus, the Hopf algebras $H$ in the theorem can be expressed as Hopf duals of
finite dimensional Hopf quotients of $\mathcal{O}_{\pm 1}(SL_2(k))$.

If $p=-1$, then $\mathcal{O}_{-p^{-1}}(SL_2)$ is commutative. Hence $K =
k\langle C \rangle$ is commutative. Thus, $H = K^{\circ} = k \tilde{\Gamma}$, a
group algebra for $\tilde{\Gamma}$ a nonabelian finite subgroup of 
$SL_2(k)$. Again, by Proposition \ref{pro4.2}, $R$ is isomorphic to
either $k[u,v]$ or $k_{-1}[u,v]$. The classical result for $k[u,v]$
\cite[Theorem 3.6.17.I]{Suzuki},
and the last case of Proposition~\ref{pro2.8}(c) for
$k_{-1}[u,v]$ complete these cases. In other words, we have
statements (a1) and (a2), respectively.

Now consider the $p = 1$ case;  we show that $R=k_{-1}[u,v]$. By way
of contradiction suppose that $R=k[u,v]$. For the coaction
$\rho(U) \subseteq U \otimes C$, choose $u$ and $v$ so that:
$$\rho(u) = u \otimes e_{11} + v \otimes e_{21}\hspace{.2in} 
\text{and} \hspace{.2in}
\rho(v) = u \otimes e_{12} + v \otimes e_{22}.$$
Now, writing $[a,b]:=ab-ba$, we have that
$$0 = \rho([u,v]) = u^2 \otimes [e_{11},e_{12}]+uv\otimes
([e_{11},e_{22}]+[e_{21},e_{12}])+v^2\otimes [e_{21},e_{22}].$$
Hence, by considering the term $u^2$, we have
$0=e_{11} e_{12}-e_{12} e_{11}$. Equation
\eqref{E4.5.7} then implies that $e_{11} e_{12} = e_{12} e_{11}=0$.
Similarly, by considering the terms $uv$ and $v^2$, one sees that
$e_{21}e_{22}=e_{22}e_{21}=0$ and $e_{11}e_{22}-e_{22}e_{11}=0$.
Applying the antipode $S$, we obtain that
$e_{11}e_{21}=e_{21}e_{11}=e_{12}e_{22} =e_{22}e_{12}=0.$
Together with the relation $e_{12}e_{21}=e_{21}e_{12}$
in ${\mathcal O}_{-p^{-1}}(SL_2)$,  we have that $K$ is commutative, 
which yields a contradiction. Therefore, $R\cong k_{-1}[u,v]$. This gives case (a3).
\end{proof}

Finite dimensional noncommutative Hopf quotients of
$\mathcal{O}_{-1}(SL_2)$ are described explicitly in \cite[Theorem~5.19]{BichonNatale}; we restate their classification in our context below.

\begin{corollary}
\label{cor4.6} The Hopf algebras $H$ appearing in case (a3) of
Theorem \ref{thm4.5} are duals of the finite dimensional
Hopf quotients of $\cal{O}_{-1}(SL_2)$. That is, $H^{\circ}$ is
isomorphic to exactly one of the Hopf algebras
$\mathcal{B}(\tilde{I})$, $\mathcal{A}(\tilde{\Gamma})$ or
$\mathcal{B}(\tilde{\Gamma})$, all of which we denote by $\mathcal{D}(\tilde{\Gamma})$ (as in the Introduction). Here, $\tilde{I}$ is the binary
icosahedral group of order 120, and $\tilde{\Gamma}$ is either the
binary tetrahedral group of order 24, the binary octahedral group of
order 48, or the binary dihedral group of order $4n$ for $n \geq 2$.
\qed
\end{corollary}

When $\tilde{\Gamma}=BD_{4n}$, further results are given in
\cite{Masuoka:cocycle}. By the remarks after \cite[Definition
3.3]{Masuoka:cocycle}, we have that $\mathcal{A}(BD_{4m})$ (which is
${\mathcal A}_{4m}$ in \cite{Masuoka:cocycle}) for
$m\geq 3$, and  $\mathcal{B}(BD_{4m})$ (which is
${\mathcal B}_{4m}$ in \cite{Masuoka:cocycle}) for $n\geq 2$, are
nontrivial. On the other hand, $\mathcal{A}(BD_{8})$ is isomorphic
to $(k BD_{8})^\circ$, hence is commutative. Therefore, not all Hopf
deformations of binary polyhedral groups are nontrivial.

\begin{proposition}
\label{pro4.7}
Let $(H,R)$ satisfy Hypothesis \ref{hyp0.3}. If $H$ is 
noncommutative, then $\dim H$ is even.
\end{proposition}

\begin{proof} Since $H$ is semisimple and noncommutative, the
$H$-module $U$ is simple by Lemma \ref{lem4.1}(c). Since the $H$-action
has trivial homological determinant, $U$ is self-dual by Corollary
\ref{cor3.3}(c). By \cite[Theorem 4]{KSZ}, $\dim H$ is even.
\end{proof}

This leads to the following conjecture.

\begin{conjecture}
\label{con4.8}
Let $R$ be an AS regular algebra generated in degree one of
Gelfand-Kirillov dimension $d$. Suppose that $H$ is a semisimple Hopf
algebra acting inner faithfully on $R$ such that $R_1$ is a simple
$H$-module. Then, $\dim H$ is divisible by $d'$ for some $2\leq d'\leq d$.
\end{conjecture}


\section{Proof of Theorem \ref{thm0.4}:
$H$ is commutative (and semisimple)}
\label{sec5}

The goal of this section is to prove Theorem \ref{thm0.4} in the case 
where a semisimple Hopf algebra $H$ acts on an algebra $R$ satisfying 
Hypothesis \ref{hyp0.2}. We assume that $R = k\langle U \rangle/(r)$ 
with $U$ a non-simple left $H$-module, so by Lemma~\ref{lem4.1}(c), 
$H$ is commutative. See Theorem~\ref{thm5.2} below for the main 
classification result in this setting. First, we need the following lemma.

\begin{lemma}
\label{lem5.1} Let $H$ be a finite dimensional Hopf algebra and
$K:=H^{\circ}$.
\begin{enumerate}
\item
If $T$ is a 1-dimensional right $K$-comodule, then $T \cong kg$ for
some grouplike element $g \in G(K)$.
\item
If $K$ coacts on $R=k\langle U\rangle/(r)$ with trivial homological
determinant, then $kr \cong k 1_K$ as $K$-comodules.
\end{enumerate}
\end{lemma}

\begin{proof} (a) Take a nonzero basis element $t$ of $T$. Now
$\rho(t) = t \otimes g$, and by coassociativity,
$$t \otimes \Delta(g) = (1 \otimes \Delta) \circ \rho(t) =
(\rho \otimes 1) \circ \rho(t) = t \otimes g \otimes g.$$
Hence, $\Delta(g) = g \otimes g$.

(b) This follows from Theorem \ref{thm2.1}.
\end{proof}

\begin{theorem}
\label{thm5.2}
Assume Hypothesis~\ref{hyp0.3}, i.e.
$H$ be a semisimple Hopf algebra acting on an AS regular algebra 
$R = k \langle U \rangle/(r)$ of global dimension 2, with trivial 
homological determinant. Assume that the left $H$-module $U$ is 
non-simple so $H$ is commutative by Lemma~\ref{lem4.1}(c). The 
pairs $(H,R)$ that occur are given as follows:
\begin{enumerate}
\item[(b1)]
$\left( kC_2, ~R \right)$ where $R = k_J[u,v]$ or $k_q[u,v]$, and the
action of the generator $\sigma\in C_2$ on $R$ is defined by
$\sigma(f)=(-1)^{\deg f} f$ for all homogeneous elements
$f\in R$;
\item[(b2)]
$\left( kC_2,~ k\langle u,v \rangle/(u^2+v^2) \right)$, and the
action of the generator $\sigma\in C_2$ is defined by $\sigma(u)=u,
\sigma(v)=-v$;
\item[(b3)]
$\left( kC_n, ~k_q[u,v] \right)$ for $n \geq 3$, and the action of a
generator $\sigma\in C_n$ is defined by $\sigma(u)=\zeta u$ and
$\sigma(v)= \zeta^{-1} v$, for some primitive $n$-th root of unity $\zeta$;
\item[(b4)]
$\left( (kD_{2n})^{\circ},~ k\langle u,v \rangle/(u^2+v^2)  \right)$
for $n \geq 3$,
and if $\{p_x\}_{x \in D_{2n}}$ is the dual basis of $D_{2n}$, then
$p_x \cdot u = \delta_{x,g} u$ and $p_x \cdot v = \delta_{x,h} v$.
Here, $D_{2n} = \langle g,h ~|~ g^2=h^2=1, ~(gh)^n=1 \rangle$.
\end{enumerate}
Note that algebra $k\langle u,v \rangle/(u^2+v^2)$ is
isomorphic to $k_{-1}[u,v]$.
\end{theorem}

\begin{proof}
Recall that $H\neq k$.
The left $H$-module $U$ is isomorphic to $T_1 \oplus T_2$ where
$T_i$ is a 1-dimensional left $H$-module by Lemma \ref{lem4.1}(a).
To prove (b1) - (b4), recall that since $H$ is commutative and $k$
is algebraically closed of characteristic zero, we have that $H$ is
isomorphic to the dual of a group algebra $(kG)^{\circ}$
\cite[Theorem~2.3.1]{Montgomery}.  In this case, we consider $R$ 
as a $G$-graded algebra. Let $K$ denote $kG$, the Hopf dual of $H$ 
and let $\rho$ denote the $K$-coaction on $R$.

\smallskip

 (b1) Assume that $U \cong T \oplus T$ for some 1-dimensional
right $K$-comodule $T\not\cong k$. Pick a basis $u$ for the first
copy of $T$, and $v$ for the second. By Lemma \ref{lem5.1}(a),
$\rho(u) = u \otimes g$ and $\rho(v) = v \otimes g$, for some $g \in
G(K)$. Since the relation of the AS regular algebra $R$ of global
dimension two is of the form $r = a u^2 + b uv + cvu +d v^2$, we
have that $\rho (r) = r \otimes g^2$.

Since hdet$_H R$ is trivial, we have that $g^2 = 1$ by Lemma
\ref{lem5.1}(b). Also, $\deg(u) = \deg(v) = g$, so $K$ is
generated by $g$ and whence $K=k\left<g ~|~ g^2\right>=k C_2$
by Lemma \ref{lem1.6}. Thus, $H = (k C_2)^{\circ} \cong k C_2$.

Since $R$ is any AS regular algebra of dimension 2, and $H \cong k
C_2 = k\left<\sigma\right>$ for some generator $\sigma$ of $C_2$,
the only possible action of $H$ on $R$ with $U\cong T\oplus T$
is given by $\sigma(u) = -u $ and $\sigma(v) = -v$. This also
follows from Proposition \ref{pro2.8}.

\smallskip

(b2) Assume that $U \cong k \oplus T$ where $T\not\cong k$.
Pick a basis $u$
for $k$, and $v$ for $T$. Then $\rho (u) = u \otimes 1$ and $\rho(v)
= v \otimes g$, for some $g \in G(K)$.  Note that
$$\rho(u^2) = u^2 \otimes 1, \hspace{.15in}
\rho(uv) = uv \otimes g, \hspace{.15in} \rho(vu) = vu \otimes g,
\hspace{.15in} \rho(v^2) = v^2 \otimes g^2.$$ Since $\hdet_H R$ is
trivial and $g \neq 1$, and since $R$ is a domain, the relation $r$
of $R$ must be of the form $u^2 + \alpha v^2$ for some
$\alpha\neq 0$. By Lemma \ref{lem5.1}(b), $g^2 =1$. Rescaling this
relation yields the desired algebra, $R =
k\left<u,v\right>/(u^2+v^2)$. Again by Lemma \ref{lem1.6}, $K =
k\left<g ~|~ g^2 =1\right> = k C_2$. Thus, $H = (k C_2)^{\circ}
\cong k C_2$. Note that $R \cong k_{-1}[u,v]$ and the non-diagonal
action of $H$ on $R$ is given in the second case $G=C_2$ in 
Proposition~\ref{pro2.8}(c). Equivalently,  the action of $H =
k\left<\sigma\right>$ on $k \langle u,v \rangle/(u^2+v^2)$ is given
by $\sigma(u) = u$ and $\sigma(v) = -v$.

\smallskip

(b3) Assume that $U \cong T_1 \oplus T_2$ for some non-isomorphic
1-dimensional left $H$-modules $T_i$, where $T_1 \otimes T_2 \cong k$.
Pick a basis $u$ for $T_1$, and $v$ for $T_2$.   By Lemma \ref{lem5.1}(a),
$\rho(u) = u \otimes g$ and $\rho(v) = v \otimes h$, for some
$g,h \in G(K)$. Here, $g \neq h$, else we are in case (b1). Note that
$$\rho(u^2) = u^2 \otimes g^2, \hspace{.15in}
\rho(uv) = uv \otimes gh, \hspace{.15in}
\rho(vu) = vu \otimes hg, \hspace{.15in}
\rho(v^2) = v^2 \otimes h^2.$$
Since $T_1 \otimes T_2 \cong k$, and $T_1 \otimes T_2 \cong kg
\otimes kh$ as $K$-comodules, we have that $gh = 1 =hg$.
This implies that $g^2 \neq 1$ and $h^2 \neq 1$. Hence, the relation
$r$ of $R$ lies in $(T_1 \otimes T_2) \oplus (T_2 \otimes T_1)$.
Rescaling $r$ implies that $R =k_q[u,v]$ for some $q \in k^{\times}$.

Note that $h = g^{-1}$, so by Lemma \ref{lem1.6}, $K$ is generated
by the grouplike element $g$ and $G(K) = \left<g ~|~ g^n\right> =
C_n$ for some $n \geq 3$. Thus, $H = (k C_n)^{\circ} \cong k C_n$.
By Proposition \ref{pro2.8}, we have that the action of $H =
k\left<\sigma\right>$ on $R$ is given by $\sigma(u) = \zeta u$ and
$\sigma(v) = \zeta^{-1} v,$ where $\zeta$ is a primitive $n$-th root
of unity.

\smallskip

(b4) Assume that $U \cong T_1 \oplus T_2$, where
$T_1 \otimes T_2 \not \cong k$. Pick a basis $u$ for $T_1$, and $v$
for $T_2$.   Again by Lemma~\ref{lem5.1}(a), $\rho(u) = u \otimes g$
and $\rho(v) = v \otimes h$, for some $g,h \in G(K)$. Here,
$g \neq h$, else we are in case (b1). Note that
$$\rho(u^2) = u^2 \otimes g^2, \hspace{.15in}
\rho(uv) = uv \otimes gh, \hspace{.15in}
\rho(vu) = vu \otimes hg, \hspace{.15in}
\rho(v^2) = v^2 \otimes h^2.$$
Since $T_1 \otimes T_2 \not \cong k$, we have that $gh \neq 1$ and
$hg \neq 1$. Since $\hdet_H R$ is trivial and  $R$ is a
domain, we get that $g^2 = h^2 =1$ by Lemma~\ref{lem5.1}(b). After
rescaling, $R=k\left<u,v\right>/(u^2+v^2)$ as desired. Note that by
Lemma~\ref{lem1.6}, $K$ is generated by $g$ and $h$. Since these are 
both grouplike elements, we have $K=k G ( K )$. Let $n$ denote an 
integer such that $(gh)^n=1$, so we have the following relations 
$g^{2}=h^{2} = ( g h )^{n} = 1$. Thus $G(K)$ is a quotient of the
dihedral group $D_{2n}$. By classical group theory, any quotient group of 
$D_{2n}$ is of the form $D_{2m}$ for some $m\mid n$. Without loss of 
generality, $G(K)=D_{2n}$, or equivalently, $K$ is the group algebra 
$k D_{2n}$. Therefore $H = (k D_{2n})^{\circ}$.

Furthermore, the right $K$-coaction on $U$ yields a left $H$-action on
$U$ as follows. Let $\{p_x\}_{x \in D_{2n}}$ be the dual basis of the
group algebra $K = k D_{2n}$, which serves as a basis of $H$. The
right coaction of $u$ is defined as $\rho(u) = u \otimes g
\in U \otimes K$, which implies that the left action of $p_x$ on
$u \in U$ is given by:
$p_x \cdot u = \left< p_x, g\right>u = \delta_{x,g} \cdot u.$
Likewise, the left coaction $\rho(v) = v \otimes h \in U
\otimes K$ implies that:
$p_x \cdot v = \left< p_x, h\right>v = \delta_{x,h} \cdot v. $
\end{proof}


\section{Proof of Theorem \ref{thm0.4}: $H$ is non-semisimple}
\label{sec6}

The main result of this section is to prove Theorem \ref{thm0.4} 
in the case that $H$ is non-semisimple. We set the 
following notation for the rest of the section.

\medskip

\begin{notation}
\label{not6.1}
We set $q$ to be a root of unity with $q^2 \neq 1$. Let $l$ be order of
$q$, let $m$ be the order of $q^2$. Note that $l=m$ if $l$ is odd, and
$l=2m$ if $l$ is even.
\end{notation}

\begin{theorem}
\label{thm6.2} 
Assume Hypothesis \ref{hyp0.2} so that
$H$ is a non-semisimple Hopf algebra that acts on an Artin-Schelter 
regular algebra $R=k\langle U \rangle/(r)$ of global dimension 2 with 
trivial homological determinant. Then the pair $(H,R)$ arises in one 
of the following cases.
\begin{enumerate}
\item[(c1)]
$(T_{q, \alpha,n}^{\circ}, k_q[u,v])$ where $T_{q,\alpha,n}$ is a
generalized Taft algebra  for $q$ a root of unity with $q^2\neq 1$
(Definition~\ref{def6.4}). In this case, the left $H$-module $U$ is
not semisimple as an $H$-module.
\end{enumerate}
If in addition,  $U$ is a simple left $H$-module, then $R=k_q[u,v]$ for $q$ a
root of unity with $q^2 \neq 1$, and one of the following pairs
occur.
\begin{enumerate}
\item[(c2)]
Assume that the order of $q$ is odd.  We have $(H,
k_{q}[u,v])$, where $H^{\circ}$ is an
$(k\tilde{\Gamma})^{\circ}$-extension of the dual of the
Frobenius-Lusztig kernel $\mathfrak{u}_q(\mathfrak{sl}_2)$, with
$\tilde{\Gamma}$ a finite subgroup of $SL_2(k)$. Uniqueness of this
extension is discussed in Proposition \ref{pro6.12}. Moreover,
$H^{\circ}$ is coquasitriangular, quasicommutative, not pointed, and
$\dim_k H = \dim_k H^{\circ} = |\tilde{\Gamma}| \cdot l^3$.
\item[(c3)]
Assume that the order of $q$ is even.  We have $(H, k_{q}[u,v])$, where $H^{\circ}$ is
an $(k\Gamma)^{\circ}$-extension of either
\begin{itemize}
\item the dual of the double
Frobenius-Lusztig kernel $\mathfrak{u}_{2,q}(\mathfrak{sl}_2)$ 
(Definition-Theorem~\ref{def6.16}) \hfill if $q^4 \neq 1$, or
\item $u_{2,q}(\mathfrak{sl}_2)^{\circ}$ or the 8-dimensional quotient $u_{2,q}(\mathfrak{sl}_2)^{\circ}/(e_{12} - e_{21}e_{11}^2)$ ~~if $q^4 =1$.
\end{itemize}
Here,  $\Gamma$ is a finite subgroup of $PSL_2(k)$.
Moreover, $H^{\circ}$ is not pointed. If $q^4 \neq 1$, then $\dim_k H = 2 |\Gamma| (l/2)^3 = 2 |\Gamma| m^3.$
\end{enumerate}
The invariant subring $R^H \cong R^{\mathrm{co}\; H^{\circ}}$ in each case is AS Gorenstein.
\end{theorem}

\begin{proof} The proof is based on analysis of these three cases, (c1) - (c3), 
which are discussed in Sections~6.1-6.3, respectively. Namely, see Proposition~\ref{pro6.6} for (c1); see Proposition~\ref{pro6.12} for (c2); and see Proposition~\ref{pro6.25} and Remark~\ref{rem6.24} for (c3). Moreover, Proposition~\ref{pro2.7}(c) verifies the statement before (c2). The AS Gorenstein condition is established in Lemma~\ref{lem6.7} and Propositions \ref{pro6.8}, \ref{pro6.14}, \ref{pro6.28}.
\end{proof}

\subsection{$U$ non-simple}
\label{ssec6.1}
In this section, we assume that $H$ is non-semisimple, and for 
$R=k\langle U\rangle/(r)$, we have that $U$ is a non-simple left 
$H$-module. 
We classify such pairs $(H,R)$ in Proposition~\ref{pro6.6} under the 
condition that the homological determinant is trivial. The coaction 
of $H^{\circ}$ on $R$ is also provided here. Moreover, we describe 
the invariant subrings $R^H$ in Lemma~\ref{lem6.7} and 
Proposition~\ref{pro6.8}.

Let $K$ denote the
Hopf dual $H^\circ$ of $H$. Furthermore, the non-simple $K$-comodule $U$ is
not the direct sum of two 1-dimensional simple modules.
(Recall that if $U$ is the direct sum of two 1-dimensional $H$-modules,
then $H$ is commutative [Lemma \ref{lem4.1}]. Hence $H$ is semisimple,
yielding a contradiction.) In this case we also call $U$ non-semisimple.
For any 2-dimensional non-semisimple $K$-comodule $U$, there is a 
non-split short exact sequence of $K$-comodules
$$0\to T_1\to U\to T_2\to 0,$$
where $T_1$ and $T_2$ are 1-dimensional $K$-comodules.
By Lemma \ref{lem5.1}(a),
$T_i\cong kg_i$ for some grouplike elements $g_i\in G(K)$, for $i=1,2$.

To classify the algebras $R = k\langle U \rangle/(r)$ that occur in
Theorem \ref{thm6.2}(c1), pick a basis $\{b_1,b_2\}$ of $U$ so that
$b_1\in T_1$ and $b_2\in U\setminus T_1$. We get that
\begin{align}
\rho(b_1)&= b_1\otimes g_1,\label{E6.2.1}\tag{E6.2.1}\\
\rho(b_2)&= b_2\otimes g_2+ b_1\otimes x \label{E6.2.2}\tag{E6.2.2}
\end{align}
for some nonzero $x\in K$. (This holds because $\overline{b_2} 
\equiv b_2 \mod T_1$
is a basis of $T_2$, so $\rho(\overline{b_2}) = \overline{b_2} \otimes g_2$.
This is equivalent to $\rho(b_2) = b_2 \otimes g_2 + b_1 \otimes x$.) By
the coassociativity of $\rho$, that is to say $(\rho \otimes 1) \circ \rho =
(1 \otimes \Delta) \circ \rho$, we have that
\begin{equation}
\label{E6.2.3}\tag{E6.2.3}
\Delta(x)=x\otimes g_2+g_1\otimes x.
\end{equation}

\begin{lemma}
\label{lem6.3}  Retain the notation above. Assume that
$$r:=a_{11}b_1^2+a_{12}b_1b_2+a_{21}b_2b_1+a_{22} b_2^2=0$$
is the quadratic relation of $R$. Then, the
following statements hold.
\begin{enumerate}
\item
$a_{22}=0$.
\item
$a_{12}a_{21}\neq 0$. So, from now on we assume that $a_{21}=1$.
\item
$g_1g_2=g_2g_1$.
\item
The homological codeterminant of the $K$-coaction on $R$ is trivial if and only if
$g_1g_2=1$.
\item
If the $K$-coaction on $R$ has trivial homological codeterminant, then
$a_{12}\neq -1$. Moreover, after a basis change of $\{b_1,b_2\}$, we
have that $a_{11}=0$ and $x g_1+a_{12}g_1x=0$.
\end{enumerate}
\end{lemma}

\begin{proof}
Since $\rho: R\to R\otimes K$ defines a comodule algebra, we have 
by (\ref{E6.2.1}) - (\ref{E6.2.2}) that
$$\begin{aligned}
\rho(b_1^2)&=(b_1^2)\otimes g_1^2,\\
\rho(b_1 b_2)&=(b_1^2)\otimes g_1 x+ (b_1 b_2)\otimes g_1g_2,\\
\rho(b_2 b_1)&=(b_1^2)\otimes x g_1+ (b_2 b_1)\otimes g_2g_1,\\
\rho(b_2^2)&=(b_1^2)\otimes x^2+ (b_1 b_2)\otimes xg_2+(b_2b_1)
\otimes g_2 x+(b_2^2)\otimes g_2^2.
\end{aligned}
$$
Since $r$ is a relation of $R$, we have that $\rho(r)=0$.

(a) Proceed by contradiction. If $a_{22}\neq 0$, then we may assume
that $a_{22}=1$. In this case, $\{b_1^2, b_1b_2,b_2b_1\}$ are
linearly independent. Using the computation above, we have that
$$\begin{aligned}
0=\rho(r)&=\rho(a_{11}b_1^2+a_{12}b_1b_2+a_{21}b_2b_1+b_2^2)\\
&=(b_1^2)\otimes (a_{11}g_1^2+a_{12}g_1 x+a_{21}x g_1+x^2-a_{11} g_2^2)\\
&\quad + (b_1b_2)\otimes
(a_{12}g_1g_2+xg_2-a_{12}g_2^2)+(b_2b_1)\otimes (a_{21}g_2g_1+g_2
x-a_{21}g_2^2).
\end{aligned}
$$
Therefore,
$$\begin{aligned}
a_{11} g_2^2&=a_{11}g_1^2+a_{12}g_1 x+a_{21}x g_1+x^2,\\
a_{12}g_2^2&=a_{12}g_1g_2+xg_2,\\
a_{21}g_2^2&=a_{21}g_2g_1+g_2 x.\\
\end{aligned}
$$
Since $g_1$ and $g_2$ are invertible, the last two equations can be
simplified to
$x=a_{12}(g_2-g_1)=a_{21}(g_2-g_1).$
Hence, $K$ is generated by grouplike elements $g_1$ and
$g_2$, and consequently, $K$ is a group algebra which is
semisimple, yielding a contradiction. Therefore, $a_{22}=0$.

(b) Since $R$ is a domain, the relation $r$ is not a product of two
factors of degree one. Therefore $a_{12}a_{21}\neq 0$ as $a_{22}=0.$
For simplicity we may assume that $a_{21}=1$ from now on.

(c) By parts (a) and (b), $a_{22}=0$ and $a_{21}=1$. An easy
computation shows that
$$\begin{aligned}
0=\rho(r)&= \rho(a_{11}b_1^2+a_{12}b_1b_2+b_2b_1)\\
&=(b_1^2)\otimes (a_{11}g_1^2+a_{12}g_1x+x g_1-a_{11}g_2g_1)
+(b_1b_2)\otimes
(a_{12}g_1g_2-a_{12}g_2g_1).\\
\end{aligned}
$$
This implies that
\begin{align}
\label{E6.3.1}\tag{E6.3.1}a_{12}g_1g_2&= a_{12}g_2g_1,\\
\label{E6.3.2}\tag{E6.3.2}a_{11}g_2g_1&= a_{11}g_1^2+a_{12}g_1x+x
g_1.
\end{align}
By part (b), $a_{12}\neq 0$. Therefore $g_1g_2=g_2g_1$.

(d) We now compute $\rho$ on the free algebra generated by $b_1$ and
$b_2$. Recall that $a_{22}=0$ and $a_{21} =1$. Using
\eqref{E6.3.1}-\eqref{E6.3.2}, we have that
$$\begin{aligned}
\rho(r)&= \rho(a_{11}b_1^2+a_{12}b_1b_2+b_2b_1)\\
&=(b_1^2) \otimes (a_{11}g_2g_1)+(b_1b_2)\otimes(a_{12}g_2g_1)+
(b_2b_1)\otimes g_2g_1\\
&=r\otimes g_2g_1.\\
\end{aligned}
$$
Thus, the homological codeterminant of $K$-coaction on 
$k\langle b_1, b_2 \rangle$ is $(g_2g_1)^{-1}$
by Theorem \ref{thm2.1}. The assertion follows.

(e) By way of contradiction, assume that $a_{12}=-1$. Then, the
relation $r$ becomes
$b_2b_1=b_1b_2-a_{11}b_1^2.$
If $a_{11}\neq 0$, then $R$ is isomorphic to $k_J[u,v]$. By
Proposition \ref{pro2.7}(a), $H$ is a group algebra. Hence $K$ is
semisimple, yielding a contradiction. If $a_{11}=0$, then $R$ is
isomorphic to $k[b_1,b_2]$. Again by Proposition~\ref{pro2.7}(b), $H$ (and
so $K$) is semisimple, yielding a contradiction. Therefore,
$a_{12}\neq -1$. Since $a_{12}\neq -1$, by change a basis (replacing
$b_2$ by $b_2+(1+a_{12})^{-1}a_{11} b_1$), we may assume that
$a_{11}=0$. The last assertion follows from \eqref{E6.3.2}.
\end{proof}

Thus if the $K$-coaction on $R$ has trivial homological determinant, 
then $R=k_q[u,v]$ for some $q \in k^{\times}$. We define a family
of Hopf algebras that coacts on $k_q[u,v]$.

\begin{definition}
\label{def6.4} Here, we define a {\it generalized Taft algebra}
$T_{q,\alpha,n}$. Let $q$ be a root of unity such that the order $m$
of $q^2$ is larger than $1$. Let $\alpha\in k$, and let $n$ be a
positive integer divisible by the order of $q$. Let $T_{q,\alpha,n}$
be an algebra generated by $g, g^{-1}$ and $x$ subject to the
relations:
$$gg^{-1}=g^{-1}g=1,\quad xg=qgx,\quad g^n=1,
\quad x^m=\alpha(g^m-g^{-m}).$$
Here, $\alpha$ is either 0 or 1: if $q^m\neq 1$, then $\alpha=0$;
and if $q^m=1$, then $\alpha$ could be 0 or 1. The $k$-vector space
dimension of $T_{q,\alpha, n}$ is $mn$ as it has a basis $\{g^i
x^j\mid 0\leq i\leq n, 0\leq j\leq m\}$.

Finally, $T_{q,\alpha,n}$ becomes a Hopf
algebra with the following coalgebra structure and antipode:
$$\Delta(g)=g\otimes g, \quad \Delta(x)=x\otimes g^{-1}+g\otimes x,
\quad \epsilon(g)=1, \quad \epsilon(x)=0, \quad S(g)=g^{-1},
\quad S(x)=-q x.$$
\end{definition}

\begin{remark} \label{rem6.5}
Write $T$ for $T_{q,\alpha,n}$.
The coradical filtration of $T$ is given by
\[
\begin{array}{lll}
C_0(T)&=kG(T)=\bigoplus_{i=0}^{n-1}k g^i,\\
C_1(T)&= kG(T)\oplus kG(T)x &=kG(T)\oplus xG(T),\\
C_s(T)&=\bigoplus_{j=0}^s kG(T) x^j &=\bigoplus_{j=0}^s x^j kG(T)
\end{array}
\]
for all $s\leq m-1$. As a consequence, if $I$ is a nonzero Hopf
ideal of $T$, then $I\cap C_1(T)\neq 0$ by \cite[Theorem~5.3.1]{Montgomery}.
\end{remark}

Let us recall briefly the quantum binomial theorem. Let $p$ be a
scalar. Let $\displaystyle{ {r \choose s}_{p}}$ denote the
$p$-binomial coefficient $\displaystyle \prod_{i=0}^{s-1} \frac{1
-p^{r-i}}{1-p^{i+1}}$. If $YX=qXY$, then
$$(X+Y)^n=\sum_{s=0}^{n} \displaystyle{ {n \choose s}_{p}} X^s
Y^{n-s},$$ which is called the quantum binomial theorem. If $p$ is a
primitive $n$-th root of unity, then $\displaystyle{ {n \choose
s}_{p}}=0$ for all $s=1,2,\cdots, n-1$. In this case,
$(X+Y)^n=X^n+Y^n$. We will apply these formulas soon. Now we
classify the pairs $(H,R)$ arising in Theorem \ref{thm6.2} with 
$U$ a non-semisimple left $H$-module.

\begin{proposition}
\label{pro6.6} Assume Hypothesis \ref{hyp0.2}, and assume that
$U$ is not semisimple as an $H$-module. Then, $R=k_q[u,v]$ for some 
root of unity $q$ with
$q^2 \neq 1$, and $H^\circ\cong T_{q,\alpha,n}$. The coaction of
$T_{q,\alpha,n}$ on $R$ is determined by
\begin{align}
\rho(u)&=u\otimes g,\label{E6.6.1}\tag{E6.6.1}\\
\rho(v)&=v\otimes g^{-1}+u\otimes x.\label{E6.6.2}\tag{E6.6.2}
\end{align}
\end{proposition}

\begin{proof} Let $K=H^{\circ}$ and retain the notation used in
Lemma \ref{lem6.3}. Let $u=b_1$ and $v=b_2$.
By Lemma \ref{lem6.3}(e), $q:=-a_{12} \neq -1$, and $R=k_q[u,v]$.
By Proposition \ref{pro2.7}(c), $q$ is a root of unity, $q\neq \pm 1$.
Hence, $m>1$.

Let $g=g_1$. By Lemma \ref{lem6.3}(d), $g_2=g^{-1}$. Then the
$K$-coaction on $k_q[u,v]$ is determined by \eqref{E6.2.1} and
\eqref{E6.2.2}. Since $K$-coaction is inner faithful, $K$ is
generated by $g^{\pm 1}$ and $x$. By Lemma \ref{lem6.3}(e),
$xg=qgx$. Since $K$ is finite dimensional, the order of $g$, denoted
by $n$, is finite. Using the relation $xg=qgx$, one sees that $n$ is
divisible by the order of $q$: $(xg)g^{n-1} = q gxg^{n-1} = q^n g^n
x$ implies that $x = q^n x$.

We now show that $K \cong T_{q, \alpha, n}$. By coassociativity,
$$\Delta(x)=x\otimes g^{-1}+g\otimes x.$$
It is easy to see that $( x\otimes g^{-1})(g\otimes x)=q^2(g\otimes
x)( x\otimes g^{-1})$. By using the quantum binomial theorem and the
fact that $q^2$ is a primitive $m$-th root of unity, we get that
$$\Delta(x^m)=x^m\otimes g^{-m}+g^m\otimes x^m$$
Hence,
$$\Delta(g^m x^m)=(g^m x^m)\otimes 1+g^{2m}\otimes g^m x^m 
\hspace{.2in} \text{and} \hspace{.2in}
(g^mx^m) g^{2m}=g^{2m}(g^mx^m).$$
If $g^mx^m$ is not in $G(K)$, then by \cite[Theorem 0.2]{WZZ},
$\GKdim K\geq 1$,
yielding a contradiction. Therefore $g^m x^m\in G(K)$. An easy
computation shows that $g^m x^m=\alpha(g^{2m}-1)$  as $g^m x^m$ is
also $(g^{2m}, 1)$-primitive for some $\alpha\in k$. Equivalently,
$x^m=\alpha(g^m-g^{-m}).$

Rescaling implies that $\alpha$ is either 0 or 1. At this point, we have
shown that $K$ satisfies all of the relations, and the coalgebra
structure, and the antipode of $T_{q,\alpha,n}$. Therefore, there is
a surjective Hopf algebra homomorphism $\pi: T_{q,\alpha,n}\to K$.
Let $\{K_i\}$ be the coradical filtration of $K$. Since
$kG(T)=kG(K)=K_0=\bigoplus_{i=0}^{n-1} kg^i$, $\{1, g,\cdots,
g^{n-1}\}$ are linearly independent. The Hopf algebra structure on
$K$ induces a $K_0$-Hopf module structure on $K_1/K_0$. By 
\cite[Theorem 1.9.4]{Montgomery},
$K_1/K_0$ is a free $K_0$-module. Since $gx\in K_1/K_0$ is a
nonzero coinvariant, $\{gx, g^2x, \cdots, g^{n}x\}$ is linear
independent modulo $K_0$. This shows that $\pi$ is injective when
restricted to $C_1(T)$ (as defined in Remark \ref{rem6.5}).
Therefore $\pi$ is injective by \cite[Theorem 5.3.1]{Montgomery}.
The assertion follows.
\end{proof}

Next, we compute the invariant subring $R^H$ for the pair $(H,R)$ in
the proposition above. Recall Notation~\ref{not6.1}. First, we 
require the following lemma.

\begin{lemma}
\label{lem6.7} Suppose that $n=l$ and $K:=T_{q,0,l}$ coacts on 
$R=k_q[u,v]$ with
coaction given  by \eqref{E6.6.1}-\eqref{E6.6.2}. Then
$$R^{\mathrm{co} K}=\begin{cases} k[u^{m},v^m] & \text{if~} l=m\\
k[a,b,c]/(b^2-ac) \text{~where~} a=u^{2m}, ~~b=u^m v^m,
~~c=v^{2m}& \text{if~} l=2m,
\end{cases}$$
which are both (AS-)Gorenstein.
\end{lemma}

\begin{proof}
Note that $(v\otimes g^{-1})(u\otimes x)=q^2
(u\otimes x)(v\otimes g^{-1}).$ By \eqref{E6.6.1}-\eqref{E6.6.2}
and the quantum binomial theorem, we have for any integers $i$ and
$j$,
$$\begin{aligned}
\rho(u^i)&= u^i\otimes g^i,\\
\rho(v^i)&= \sum_{s=0}^i {i \choose s}_{q^2}(u\otimes x)^s
(v\otimes g^{-1})^{i-s},\\
\rho(u^i v^j)&=(u^i\otimes g^i)
\sum_{s=0}^j {j \choose s}_{q^2}(u\otimes x)^s
(v\otimes g^{-1})^{j-s}\\
&=(u^iv^j\otimes g^{i-j})+(u^i\otimes g^i)
\sum_{s=1}^j {j \choose s}_{q^2}(u\otimes x)^s
(v\otimes g^{-1})^{j-s}.
\end{aligned}
$$

\noindent \underline{Case 1}: $l=m$. In this case, $l$ is odd. Then
\[
\begin{array}{lll}
\rho(u^m)&=u^m\otimes g^m &=u^m\otimes 1,\\
\rho(v^m)&=\sum_{s=0}^m {m \choose s}_{q^2} (u\otimes x)^s
(v\otimes g^{-1})^{m-s}
=u^m\otimes x^m+v^m\otimes g^{-m}&=v^m\otimes 1.
\end{array}
\]
Thus $u^m, v^m\in R^{co K}$. If $f=\sum_{i+j=d} c_{ij} u^iv^j$ is
any homogeneous element in $R^{co K}$, then $\rho(f)=f\otimes 1$. If
$c_{ij}\neq 0$, then $\rho(f)=f\otimes 1$ implies that $g^{i-j}=1$,
or $i \equiv j \mod m$. Hence, we can rewrite $f$ as
$$f=\sum_{j=0}^{m-1} u^j v^j f_j(u^m,v^m)$$
where $f_j(u^m,v^m)$ is a polynomial of $u^m$ and $v^m$. In this
setting,
$$\begin{aligned}
\rho(f)&=\sum_{j=0}^{m-1} \left\{(u^jv^j\otimes 1)+(u^j\otimes g^j)
\sum_{s=1}^j {j \choose s}_{q^2}(u\otimes x)^s
(v\otimes g^{-1})^{j-s}\right\} (f_j(u^m,v^m)\otimes 1)
&=f\otimes 1,
\end{aligned}
$$
which forces
$$\sum_{j=0}^{m-1} \{(u^j\otimes g^j) \sum_{s=1}^j
{j \choose s}_{q^2}(u\otimes x)^s (v\otimes g^{-1})^{j-s}\}
(f_j(u^m,v^m)\otimes 1)=0.$$
Let $j>0$ be the maximal integer such
that $f_{j}(u^m,v^m)\neq 0$. By looking at the largest degree of $x$ in
the second tensor component, we have that
$$(u^j\otimes g^j) (u\otimes x)^{j} (f_j(u^m,v^m)\otimes
1)=0$$ which implies that $f_j(u^m,v^m)=0$, a contradiction.
Therefore $f_j(u^m,v^m)=0$ for all $j\neq 0$, and hence $f\in
k[u^m,v^m]$. The assertion follows.
\smallskip

\noindent \underline{Case 2}: $l=2m$. Then $g^{2m}=1$ (but $g^m\neq
1$). As in Case 1, one sees that $\rho(u^m)=u^m\otimes g^m$ and
$\rho(v^m)=v^m\otimes g^{-m}$. Hence $u^{2m}, v^{2m}$, and $u^m v^m$
are in $R^{\mathrm{co} K}$. As in the proof of Case 1, if $f=\sum_{i+j=d}
c_{ij} u^iv^j$ and if $c_{ij}\neq 0$, then $g^{i-j}=1$, or $2m$
divides $i-j$. Hence we can rewrite $f$ as
$$f=\sum_{j=0}^{m-1} u^j v^j f_j(u^{2m},v^{2m}, u^m v^m)$$
where $f_j(u^{2m},v^{2m},u^mv^m)$ is a polynomial of $u^{2m}$,
$v^{2m}$ and $u^m v^m$. The rest of the argument is similar to the
proof of Case 1.
\end{proof}

Now we consider the case where $n \neq l$.

\begin{proposition}
\label{pro6.8} Consider the pair $(H,R)$ under the hypothesis of
Theorem \ref{thm6.2}(c1). By Proposition \ref{pro6.6}, $H^{\circ}
\cong T_{q,\alpha,n} =:K$. Suppose that $K$ coacts on $R=k_q[u,v]$
for $n \neq l$ with coaction given by \eqref{E6.6.1}-\eqref{E6.6.2}.
Then, $R^H \cong R^{\mathrm{co} K}\cong k[a,b,c]/(ac-b^{n/m})$, which is
(AS-)Gorenstein.
\end{proposition}

\begin{proof} Retain the notation used in Lemma \ref{lem6.7}.
It is clear that there is a surjective Hopf homomorphism $K\to
K/(g^{l}-1)\cong T_{q, 0, l}$. So $R^{\mathrm{co} K}$ is a subring of $R^{\mathrm{co}
T_{q,0,l}}$. We will show that there is an induced $K$-coaction on
$R^{\mathrm{co} T_{q,0,l}}$ and $R^{\mathrm{co} K}= (R^{\mathrm{co} T_{q,0,l}})^{\mathrm{co} K}$. Again,
we divide the proof into two cases as in Lemma \ref{lem6.7}.
\smallskip

\noindent \underline{Case 1}: $l=m$. By Lemma \ref{lem6.7}, $R^{\mathrm{co}
T_{q,0,l}}=k[u^m,v^m]$. By computation,
$$\begin{aligned}
\rho(u^m)&=u^m\otimes g^m,\\
\rho(v^m)&=u^m\otimes x^m+v^m\otimes g^{-m}
&= u^m\otimes \alpha(g^{m}-g^{-m})+v^m\otimes g^{-m}.
\end{aligned}
$$
Hence, $k[u^m,v^m]$ is a $K_0$-comodule algebra, where $K_0$ is the
subalgebra of $K$ generated by $g^m$. Then, $K_0$ is the cyclic 
group algebra $k C_w$
where $w=n/m$, and
$\rho(v^m-\alpha u^m)=(v^m-\alpha u^m)\otimes g^{-m}.$
Hence $k[u^m,v^m]^{\mathrm{co} K_0}=k[a,b,c]/(ac-b^w)$ where $a=(u^m)^w$,
$b=u^m(v^m-\alpha u^m)$ and $c=(v^m-\alpha u^m)^w$. Therefore $$R^{\mathrm{co}
K} = (R^{\mathrm{co} T_{q,0,l}})^{\mathrm{co} K_0} =k[a,b,c]/(ac-b^{w}).$$

\noindent \underline{Case 2}: $l$ is even and $l=2m$. By Definition
\ref{def6.4},
 $\alpha=0$ and $x^m=0$. Then, we have
$\rho(u^m)=u^m\otimes g^m$ and
$\rho(v^m)=v^m\otimes g^{-m}$.
Consequently,
$$
\rho(u^{2m})=u^{2m}\otimes g^{2m},\quad
\rho(u^m v^m)=u^m v^m\otimes 1,\quad
\rho(v^{2m})=v^{2m}\otimes g^{-2m}.
$$
Let $K_0$ be the subalgebra generated by $g^{2m}$. Then $K_0=k C_w$
where $w=n/2m$. It is clear that 
$$(R^{\mathrm{co} T_{q,o,l}})^{\mathrm{co} K_0} =
(k[u^{2m}, u^m v^m, v^{2m}])^{\mathrm{co} K_0}=k[a,b,c]/(ac-b^{2w}),$$ 
where $a=(u^{2m})^w$, $b=u^mv^m$ and $c=(v^{2m})^w$.
\end{proof}

\subsection{$U$ simple;  order of $q$ odd}
\label{ssec6.2} Here, we classify the pairs $(H,R)$ in Theorem
\ref{thm0.4} for $H$ non-semisimple and $U$ a simple left $H$-module.
We will show that $H$ is a finite Hopf algebra quotient of
$\mc{O}_q(SL_2(k))$ for $q$ a root of unity with $q^2 \neq 1$. We assume that
$l:=$ord($q$) is odd. In this case, $l=m=\mathrm{ord}(q^2)$.

Recall the presentation of the quantum special linear group 
${\mathcal O}_q(SL_2(k))$ generated by 
$\{e_{11},e_{12},e_{21},e_{22}\}$ from Example~\ref{ex1.4}.
Let $L$ be the subalgebra of
${\mathcal O}_q(SL_2(k))$ generated by $\{e_{ij}^m\}$ for $1\leq
i,j\leq 2$.

\begin{lemma}
\label{lem6.9} \cite[Proposition III.3.1]{book:BrownGoodearl} Retain
the notation above. We have the following statements.
\begin{enumerate}
\item
$L$ is a Hopf subalgebra of ${\mathcal O}_q(SL_2(k))$.
\item
$L$ is in the center of ${\mathcal O}_q(SL_2(k))$.
\item
$L\cong {\mathcal O}(SL_2(k))$ as Hopf algebras. \qed
\end{enumerate}
\end{lemma}

Now we define a key object for the study of $H$ (or of $H^{\circ}$, in
particular).

\begin{definitiontheorem} \cite[Theorems III.7.10-11]{book:BrownGoodearl}
\label{def6.10} The kernel of the quantum analogue of the Frobenius
map, $SL_q(2) \rightarrow SL_1(2)$, is the finite dimensional Hopf
algebra:
$${\mathcal O}_q(SL_2(k))/({\mathcal O}_q(SL_2(k))L^+)=
{\mathcal O}_q(SL_2(k))/(e_{11}^m-1,~e_{12}^m,~e_{21}^m,~e_{22}^m-1).$$
This is isomorphic to the dual ${\mathfrak u}_q(\mf{sl}_2)^{\circ}$ of
the {\it Frobenius-Lusztig kernel} of $\mathfrak{sl}_2$ at $q$.
Here, $\dim_k {\mathfrak u}_q(\mf{sl}_2)^{\circ} = m^3$.
\end{definitiontheorem}

To classify the pairs $(H,R)$ in Theorem \ref{thm6.2}(c2), let
$K:=H^{\circ}$, and let $C$ be the smallest subcoalgebra of $K$ so
that $\rho(U) \subseteq U \otimes C$. Here, $U$ is simple, so the
coalgebra $C$ is also simple and isomorphic to the matrix coalgebra
$M_2(k)$. Since the $H$-action on $R$ has trivial homological
determinant, we have that $C$ is $S$-invariant by Corollary
\ref{cor3.3} and Lemma \ref{lem3.4}. Now \cite[Theorem 1.5]{Stefan}
implies that $K$ is a quotient Hopf algebra of
$\mc{O}_{-\omega^{-1}}(SL_2(k))$ where $\omega$ is a root of unity with
ord($\omega^2$) = ord($S^2|_C$) $>1$. Let $q:=-\omega^{-1}$. Now 
to classify $K=H^{\circ}$ for $(H,R)$ in Theorem~\ref{thm6.2}(c2), it
suffices to classify finite Hopf algebra quotients of
$\mc{O}_q(SL_2(k))$, for $q$ a root of unity, not equal to $\pm 1$.
We proceed as follows.

If $\mc{O}_q(SL_2(k))$ is generated by
$\{e_{ij}\}_{1 \leq i,j \leq 2}$,  then we have an inclusion of a
central Hopf subalgebra $\mc{O}(SL_2(k))$ into $\mc{O}_q(SL_2(k))$
given by $E_{ij} \overset{\iota_1}{\mapsto} e_{ij}^m$. Moreover, 
consider the following diagram:

\[
\xymatrix{
k \ar[r] &  \mc{O}(SL_2(k)) \ar@{^{(}->}[r]^{\iota_1}
\ar@{->>}[d]^{\eta_1} & \mc{O}_q(SL_2(k)) \ar@{->>}[r]^{\pi_1}
\ar@{->>}[d]^{\eta_2} & \mathfrak{u}_q(\mathfrak{sl}_2)^{\circ}
\ar[r] \ar@{->>}[d]^{\eta_3} & k\\
k \ar[r]&  K' \ar@{^{(}->}[r]^{\iota_2}    & K \ar@{->>}[r]^{\pi_2}
&  \mathfrak{u}_q(\mathfrak{sl}_2)^{\circ}/I  \ar[r] & k. }
\]
\begin{center} Diagram 3 \end{center}
\medskip

\noindent Here, $K' \cong (k\tilde{\Gamma})^{\circ}$, for
$\tilde{\Gamma}$ a finite subgroup of $SL_2(k)$, as $\im (\eta_1)$
is a commutative finite Hopf algebra quotient of $\mc{O}(SL_2(k))$.

\begin{lemma} \label{lem6.11}
The Hopf ideal $I$ of $\mf{u}_q(\mf{sl}_2)^{\circ}$ appearing in
Diagram 3 is (0).
\end{lemma}

\begin{proof}
Let $x_{i j} \in K$ be the image of $e_{i j} \in \mathcal{O}_{q} ( S L_{2} ( k
) )$. We will abuse notation and let $e_{i j}$ denote also its image in
$\mathfrak{u}_{q} ( \mathfrak{s l}_{2} )^{\circ}$. If $e_{12} \in I$, then by
\cite[Proposition 
4.9]{Muller} we have $x_{12} =0$. 
This is a contradiction since $x_{12}$ is a generator for $K$,
so we conclude that $e_{12} \nin I$. Similarly, $e_{21} \nin I$. By \cite[Theorem 4.2]{Muller}, any
nonzero Hopf ideal $I$ of $\mathfrak{u}_{q} ( \mathfrak{s l}_{2} )^{\circ}$
contains $e_{12}$ or $e_{21}$. Therefore, $I=0$. 
\end{proof}

Therefore we have proved part of the following proposition.

\begin{proposition} \label{pro6.12}
Assume Hypothesis \ref{hyp0.2} with $H$ non-semisimple and assume 
that the left $H$-module $U$ is simple. We have that $H^{\circ}$ 
is a finite Hopf algebra quotient of $\mc{O}_q(SL_2(k))$ that 
coacts on $k_q[u,v]$ for $q$ a root of unity with $q^2 \neq 1$.

If $\mathrm{ord}(q)$ is odd, then $H^{\circ}$ fits into
the following exact sequence of Hopf algebras:
\begin{equation} \label{E6.12.1} \tag{E6.12.1}
k \rightarrow (k\tilde{\Gamma})^{\circ} \rightarrow H^{\circ}
\rightarrow \mf{u}_q(\mf{sl}_2)^{\circ} \rightarrow k,
\end{equation}
for $\tilde{\Gamma}$ a finite subgroup of $SL_2(k)$. The uniqueness
of the $(k\tilde{\Gamma})^{\circ}$-extension by
$\mf{u}_q(\mf{sl}_2)^{\circ}$ is given by Table 4 below.
Moreover, $H^{\circ}$ is coquasitriangular, quasicommutative, not
pointed, and $\dim_k H = \dim_k H^{\circ} = |\tilde{\Gamma}|l^3
= |\tilde{\Gamma}|m^3$.
\end{proposition}

\[
\begin{array}{|cc|cc|}
\hline
\tilde{\Gamma} \leq SL_2(k) &&& K=H^{\circ} \text{~unique?}\\
\hline
C_n &&& \text{Yes if and only if $(m,n) = 1$}\\
BD_{4n} &&& \text{Yes}\\
E_6 &&& \text{Yes if and only if $(m,3) = 1$}\\
E_7 &&& \text{Yes}\\
E_8 &&& \text{Yes}\\
\hline
\end{array}
\]
\medskip

\begin{center} Table 4: The uniqueness of $H^{\circ}$ \end{center}
\medskip

\begin{proof}
By Proposition \ref{pro2.7}(c),  $R$ must be isomorphic to $k_{q}[u,v]$ 
where $q$ is a root of unity, not equal to $\pm 1$. By 
\cite[Proposition 5.5]{CWZ:Nakayama}, $H^{\circ}$ 
is a Hopf algebra quotient of $\mc{O}_q(SL_2(k))$ that 
coacts on $k_q[u,v]$. 
The Hopf algebra $H^{\circ}$ is coquasitriangular, quasicommutative,
and is not pointed \cite[Lemma 5.1 and Theorem 5.6]{Muller}. We see
from Diagram 3 that $\dim H^{\circ}= |\tilde{\Gamma}| \cdot m^3$.

By \cite[Theorem 4.11]{Muller}, the
$(k\tilde{\Gamma})^{\circ}$-extensions by
$\mf{u}_q(\mf{sl}_2)^{\circ}$ are in one-to-one correspondence with
extensions $\lambda: \mathbb{Z} \rightarrow \chi(\tilde{\Gamma})$ of
$\kappa: m\mathbb{Z} \rightarrow \chi(\tilde{\Gamma})$. Here,
$\kappa(m z) = 1$ for all $z \in  \mathbb{Z}$ by \cite[Equation
4]{Muller}.

\[
\xymatrix{
\mathbb{Z} \ar@{-->}[rd]^{\lambda} & \\
m \mathbb{Z} \ar@{^{(}->}[u]  \ar[r]^{\kappa} & \chi(\tilde{\Gamma})
}
\]

Note that $\chi(C_n) \cong C_n$ as follows. If $C_n = \langle \sigma
~|~ \sigma^n=1 \rangle$, then $\chi(C_n) = \{\chi_{j}\}_{1\leq j
\leq n}$ where $\chi_j(\sigma) = \zeta^j$ for $\zeta$ some $n$-th
root of unity. Considering this group isomorphism, we have that
$\lambda(1) = \zeta^j$ for some $j$. For $\lambda$ to be an
extension of $\kappa$, we require that $1 = \kappa(m) = \lambda(m) =
\zeta^{jm}$. Therefore, $\lambda$ is unique if and only if $n$ is
coprime to $m$. In this case, $\lambda$ is the trivial map.

Moreover, $\chi(BD_{4n}) \cong C_4$ as follows. Say that $BD_{4n} =
\langle \sigma, \tau ~|~ \sigma^{2n}=1, \tau^2 = \sigma^n, \sigma
\tau = \sigma^{-1} \tau \rangle.$ Then,  $\chi(BD_{4n}) =
\{\chi_{j}\}_{0 \leq j \leq 3}$ where $\chi_{j}(\sigma, \tau) =
((-1)^j, (\sqrt{-1})^j)$. Now if $\lambda$ defined by $\lambda(1)  =
((-1)^j, (\sqrt{-1})^j)$ is an extension of $\kappa$, then $(1,1) =
\kappa(m) = \lambda(m) = ((-1)^{mj}, (\sqrt{-1})^{mj})$. Since $m$
is odd, $j=0$. Therefore $\lambda$ is the trivial map, and is
unique.

The character groups of $E_6$, $E_7$, and $E_8$ are isomorphic to
$C_3$, $C_2$, and $C_1$ respectively. Hence, we apply the argument
from the cyclic group case to yield the result.
\end{proof}

Now we study the invariant subring $R^H \cong R^{\mathrm{co} H^{\circ}}$ for
the pair $(H,R)$ in the proposition above. By Equation
(\ref{E6.12.1}), $R^{\mathrm{co} H^{\circ}} = (R^{\mathrm{co}\;
\mf{u}_q(\mf{sl}_2)^{\circ}})^{\mathrm{co} (k\tilde{\Gamma})^{\circ}}$ for $R
= k_{q}[u,v]$ with $q$ a root of unity for $q^2 \neq 1$.

\begin{lemma}
\label{lem6.13} The coinvariant subring $R^{\mathrm{co}\;
\mf{u}_q(\mf{sl}_2)^{\circ}}$ of $R=k_q[u,v]$ is isomorphic
to $k[u^m, v^m]$. 
\end{lemma}

\begin{proof}
By the quantum binomial theorem, the coaction of $\mc{O}_q(SL_2(k))$ 
on $k_q[u,v]$ (from Example~\ref{ex1.4}) yields:
\begin{equation} \label{E6.13.1} \tag{E6.13.1}
\rho(u^m)= u^m\otimes e_{11}^m+ v^m\otimes e_{21}^m,\quad
\rho(v^m)= u^m\otimes e_{12}^m+ v^m\otimes e_{22}^m.
\end{equation}

Let $\bar{\rho}:R\to R\otimes\mf{u}_q(\mf{sl}_2)^{\circ}$ be the 
induced coaction of $\mf{u}_q(\mf{sl}_2)^{\circ}$ on $k_q[u,v]$. Then by 
the formulas above, $\bar{\rho}(u^m) = u^m\otimes 1$ and 
$\bar{\rho}(v^m) = v^m\otimes 1$.
Hence, $k[u^m, v^m]\subset R^{\mathrm{co}\; \mf{u}_q(\mf{sl}_2)^{\circ}}$.
Consider the factor ${\mathfrak u}_q(\mf{sl}_2)^{\circ}/(e_{21})$,
which is isomorphic to the Hopf algebra $T_{q,0,m}$ in Definition
\ref{def6.4}. Now by Lemma \ref{lem6.7},
$R^{\mathrm{co}\; \mf{u}_q(\mf{sl}_2)^{\circ}} \subseteq
R^{\mathrm{co}\; T_{q,0,m}} \subseteq k[u^m,v^m],$
so the assertion holds.
\end{proof}

\begin{proposition} 
\label{pro6.14}
Let $(H,R)$ be a pair satisfying Hypothesis \ref{hyp0.2} in the
setting of Theorem \ref{thm6.2}(c2). Then, the invariant subring
$R^H$ of $R$ is (AS-)Gorenstein.
\end{proposition}

\begin{proof}
From the sequence, $k \to (k \tilde{\Gamma})^{\circ} \to H^{\circ}
\to \mf{u}_q(\mf{sl}_2)^{\circ} \to k$, we have that
$$R^H = R^{\mathrm{co}\;H^{\circ}} = \left( R^{\mathrm{co}\;
\mf{u}_q(\mf{sl}_2)^{\circ}} \right)^{\mathrm{co}
(k\tilde{\Gamma})^{\circ}}.$$ By Lemma \ref{lem6.13}, $R^{\mathrm{co}\;
\mf{u}_q(\mf{sl}_2)^{\circ}} \cong k[u^m, v^m]$ where $l=m$ is the
order of $q$ (or order of $q^2$).
By Lemma \ref{lem6.9}, the subalgebra of $\cal{O}_q(SL_2(k))$ 
generated by $\{e_{ij}^m\}$ is isomorphic to $\cal{O}(SL_2(k))$.
We see from \eqref{E6.13.1} that $(k{\tilde{\Gamma}})^{\circ}$ 
coacts inner faithfully 
on $k[u^m,v^m]$. Equivalently ${\tilde{\Gamma}}$ acts faithfully 
on $k[u^m,v^m]$.
Since ${\tilde{\Gamma}}$ is a finite subgroup of $SL_2(k)$, we conclude by 
\cite[Theorem 1]{Watanabe} that $R^H$ is (AS-)Gorenstein.
\end{proof}


\subsection{$U$ simple; order of $q$ even}
\label{ssec6.3} We continue to classify the pairs
$(H,R)$ in Theorem \ref{thm0.4} for $H$ non-semisimple and $U$ a
simple left $H$-module. By Proposition~\ref{pro2.7}(c),
we have that $R \cong k_{q}[u,v]$ where $q$ is a root of unity
with $q^2 \neq 1$. Recall Notation~\ref{not6.1}. We assume here that 
$2m=l=:\mathrm{ord}(q)$ is even. By 
\cite[Proposition~5.5]{CWZ:Nakayama}, $H^{\circ}$ is a finite Hopf 
algebra quotient of $\mc{O}_q(SL_2(k))$ that coacts on $k_q[u,v]$. 
We will show that $H$ is as described in Theorem \ref{thm6.2} (c3). Since $U$ is a two-dimensional simple $H$-module, $K:=H^{\circ}$ is not pointed.

As $l=2m$, consider  the subalgebra
$N$ of $\mc{O}_q(SL_2(k))$  generated by $\{e_{ij}^m
e_{st}^m\}$ for $1 \leq i,j,s,t \leq 2$.

\begin{lemma}
\label{lem6.15}  Retain the notation above. We have the following statements.
\begin{enumerate}
\item
$N$ is a normal Hopf subalgebra of ${\mathcal O}_q(SL_2(k))$.
\item
$N\cong {\mathcal O}(PSL_2(k))$ as Hopf algebras.
\end{enumerate}
\end{lemma}

\begin{proof}
Both parts follows from \cite[I.7.7]{BrownGoodearl} and routine computations.
\end{proof}

Now we define a key object for the study of $H$ (or of $H^{\circ}$
in particular). This algebra was first introduced in
\cite{Takeuchi}; we refer to it by a different name.

\begin{definitiontheorem} \cite[Definition 5.4.1]{Takeuchi}
\label{def6.16} Let $q$ be a root of unity with $q^2 \neq 1$, with
 $\mathrm{ord}(q)=2m$ even. Consider the quotient Hopf algebra:
$${\mathcal O}_q(SL_2(k))/({\mathcal O}_q(SL_2(k))N^+)=
{\mathcal
O}_q(SL_2(k))/(e_{11}^{2m}-1,e_{12}^m,e_{21}^m,e_{22}^{2m}-1).$$ We
denote this Hopf algebra by ${\mathfrak
u}_{2,q}(\mf{sl}_2)^{\circ}$, and we call this  the {\it dual of the
double Frobenius-Lusztig kernel} of $\mathfrak{sl}_2$ at $q$.
Moreover, $\dim_k {\mathfrak u}_{2,q}(\mf{sl}_2)^{\circ} = 2m^3$ by
\cite[Section~5.5]{Takeuchi}.
\end{definitiontheorem}

To study the Hopf algebra $H$ in the pair $(H,R)$ in Theorem
\ref{thm6.2}(c3), we proceed as follows. Again, we assume that
$\mathrm{ord}(q)=2m$ is even. If $\mc{O}_q(SL_2(k))$ is generated by
$\{e_{ij}\}_{1 \leq i,j \leq 2}$,  then we have an inclusion of a
normal Hopf subalgebra $\mc{O}(PSL_2(k))$ into $\mc{O}_q(SL_2(k))$
given by $E_{ij}E_{st} \overset{\iota_1}{\mapsto} e_{ij}^m
e_{st}^m$. Moreover, we can consider the following diagram:

\[
\xymatrix{
k \ar[r] &  \mc{O}(PSL_2(k)) \ar@{^{(}->}[r]^{\iota_1}
\ar@{->>}[d]^{\eta_1} & \mc{O}_q(SL_2(k)) \ar@{->>}[r]^{\pi_1}
\ar@{->>}[d]^{\eta_2} & \mathfrak{u}_{2,q}(\mathfrak{sl}_2)^{\circ}
\ar[r] \ar@{->>}[d]^{\eta_3} & k\\
k \ar[r]&  K' \ar@{^{(}->}[r]^{\iota_2}    & K \ar@{->>}[r]^{\pi_2}
&  \mathfrak{u}_{2,q}(\mathfrak{sl}_2)^{\circ}/I  \ar[r] & k. }
\]
\begin{center} Diagram 5 \end{center}
\bigskip

\noindent Here, $K' \cong (k\Gamma)^{\circ}$, for $\Gamma$ a finite
subgroup of $PSL_2(k)$, as im$(\eta_1)$ is a commutative, finite
Hopf algebra quotient of $\mc{O}(PSL_2(k))$. 
As before, we will denote $x_{ij}=\eta_2(e_{ij})$ and use $e_{ij}$ to denote also their images
in $\mf{u}_{2,q}(\mf{sl}_2)^{\circ}$. 
We will show that $I$ in 
the diagram above is $(0)$ when $q^4\neq 1$. Consider the next four lemmas.

\begin{lemma}
 \label{lem6.17}
Let $W$ be the Hopf algebra $\mf{u}_{2,q}(\mf{sl}_2)^{\circ}$. Then, 
we have the following statements.
\begin{enumerate}
\item 
$W\cong k\langle e_{11}, e_{12},e_{21}\rangle/(R)$ where the
relation ideal $(R)$ is generated by
$$e_{12}e_{11}-q e_{11}e_{12}, \hspace{.2in} 
e_{21}e_{11}-qe_{11}e_{21},\hspace{.2in} 
e_{21}e_{12}-e_{12}e_{21}, \hspace{.2in} e_{11}^{2m}-1, 
\hspace{.2in} e_{12}^m, \hspace{.2in} e_{21}^m,
$$
and $W$ has a $k$-linear basis 
$$\{e_{11}^i e_{12}^j e_{21}^l\mid
0\leq i\leq 2m-1, ~0\leq j\leq m-1, ~0\leq l\leq m-1\}.$$
\item
$W$ is an ${\mathbb N}$-graded algebra with $\deg(e_{11})=0$
and $\deg(e_{12})=\deg(e_{21})=1$.
\item
Any ideal of $W$ is ${\mathbb N}$-graded.
\end{enumerate}
\end{lemma}

\begin{proof}
(a)  Using the relations $e_{11}^{2m}=1$ and 
$e_{11}e_{22}-q^{-1}e_{12}e_{21}=1$ (see Example 1.4), one can write 
$e_{22}$ in terms of $e_{11}, e_{12}$, and $e_{21}$. Hence $W$ is 
generated by $e_{11}, e_{12}, e_{21}$. All relations
listed above are satisfied by $W$. Hence, there is a 
surjective algebra homomorphism from $k\langle e_{11}, e_{12},e_{21}
\rangle/(R)\to W$. By Definition-Theorem~\ref{def6.16}, $\dim W=2m^3$.
Moreover, it is clear that $k\langle e_{11}, e_{12},e_{21}
\rangle/(R)$ is spanned by monomials 
$$\{e_{11}^i e_{12}^j e_{21}^l\mid
0\leq i\leq 2m-1, 0\leq j\leq m-1, 0\leq l\leq m-1\}.$$
So, the assertions follow.

(b) Since all relations are homogeneous, $W$ is graded. 

(c) For any element $f$ in $W$, let $f_i$ denote the degree $i$
component of $f$. Then $f=f_0+f_1+\cdots+ f_{2m-1}$. 
Let $\eta$ be the conjugation by $e_{11}$, namely, $\eta:
f\to e_{11}^{-1} fe_{11}$ for all $f\in W$. Using relations
$e_{12}e_{11}-q e_{11}e_{12}=e_{21}e_{11}-qe_{11}e_{21}=0$,
one sees that $\eta(f_i)=q^i f_i$ for all $i$. Applying $\eta^j$
to $f$ we have 
$$\eta^j(f)=\eta^j(f_0+f_1+\cdots+ f_{2m-1})=
f_0+ (q^j)^1 f_1+ \cdots +(q^j)^{2m-1} f_{2m-1}.$$
If $f$ is in an ideal $I$ of $W$, so is $\eta^j(f)=e_{11}^{-j}fe_{11}^{j}$
for all $j$. Thus we have $f_0+ (q^j)^1 f_1+ \cdots +(q^j)^{2m-1} 
f_{2m-1} \in I$ for all $j$. Taking $j=0,\cdots, 2m-1$, we obtain
that
$$M  \begin{pmatrix}
     f_0, f_1, \dots , f_{2m-2}, f_{2m-1}
    \end{pmatrix}^T
    \in I^{\oplus n},$$
    where
$M$ is the $(2m)\times (2m)$-matrix 
$\left((q^i)^{j}\right)_{0\leq i,j \leq 2m-1}$.

Since $q^i\neq q^l$ for any $0\leq i<l\leq 2m-1$, $M$ is invertible
and $f_i\in I$ for all $i$. Therefore, $I$ is graded.
\end{proof}

\begin{lemma}
\label{lem6.18}
Let $B$ be a Hopf quotient of $W$. Let $y_{ij}$ denote the image
of $e_{ij}$ in $B$.
\begin{enumerate}
\item 
If $y_{12}\neq 0$ (or $y_{21}\neq 0$), then $y_{11}^k\neq 1$ for 
all $0<k<2m$.
\item
If $y_{12}\neq 0$, then $y_{12}^k\neq 0$ for all $0< k<m$.
\item
Let $I$ be nonzero Hopf ideal of the Hopf quotient $B=W/(e_{21})$.
Then $y_{12}\in I$, or equivalently, $e_{12}\in I+(e_{21})$.
\end{enumerate}
\end{lemma}

\begin{proof} (a) Observe that $y_{12}(y_{11}^k-1)=(q^k y_{11}^k-1)y_{12}$.
If $y_{11}^k=1$ for some $0<k<2m$, the above equation implies that
$0=y_{12}(q^k-1)$. Since $q^k-1\neq 0$, we have that $y_{12}=0$, a 
contradiction.

(b) In the proof below we will use the image of $e_{22}$, denoted by
$y_{22}$. By the definition of $W$, $y_{22}=y_{11}^{-1}(1+qy_{12}y_{21})$. 
Let $N\geq 2$ be the minimal positive integer such that $y_{12}^N=0$.
Suppose $N<m$. Then we have 
$$0=\Delta(y_{12}^N)=(y_{11}\otimes y_{12}+y_{12}\otimes y_{22})^N=
\sum_{i=0}^{N} \binom{N}{i}_{q^{-2}} y_{12}^iy_{11}^{N-i}
\otimes y_{22}^i y_{12}^{N-i}.
$$
Using the hypothesis $y_{12}^N=0$, the above equation becomes
$$0=\sum_{i=1}^{N-1} \binom{N}{i}_{q^{-2}} y_{12}^iy_{11}^{N-i}
\otimes y_{22}^i y_{12}^{N-i}.
$$
Now the elements $\{\binom{N}{i}_{q^{-2}} y_{12}^iy_{11}^{N-i}
\mid i=1,\cdots, N-1\}$
are nonzero homogeneous elements of distinct degrees. Hence, these are linearly
independent. Thus $y_{22}^i y_{12}^{N-i}=0$ for all $i=1,\cdots,N-1$.
Since $y_{22}$ is invertible, $y_{12}^{N-i}=0$, a contradiction. 
Therefore $N=m$
and the assertion follows.

(c) Consider the Hopf quotient $\bar{B}=B/I=W/(e_{21}, I)$. Suppose 
$y_{12}\not\in I$ (or $y_{12}\neq 0$ in $\bar{B}$). 
The coradical $\bar{B}_0$ of $\bar{B}$ is spanned by the grouplike elements
$1, y_{11}, \ldots, y_{11}^{2 m - 1}$. By part (a), these are distinct, 
we have $\bar{B}_0 \simeq k C_{2 m}$. Now $\bar{B}$ is pointed, so
the coradical filtration is given by (c.f. \cite[Theorem 5.4.1]{Montgomery})
\begin{eqnarray*}
    \bar{B}_i & = & \bar{B}_0 \oplus \bar{B}_0 y_{12} \oplus \cdots \oplus
    \bar{B}_0 y^i_{12}
\end{eqnarray*}
Here, $y_{12} \not \in k(y_{11} - y_{11}^{-1})$, 
which follows 
since $y_{12}^m=0$ for $m > 1$. By part (b), none of $y_{12}, 
\ldots, y_{12}^{m - 1}$ are zero, so we see that
$\bar{B}_{m - 1} = \bar{B}$ so $\dim_k \left( \bar{B} \right) = \dim_k
\left( \bar{B}_{m - 1} \right) = 2 m^2$. However, $\dim_k \left( B \right)
= 2 m^2$ as well, hence $I= 0$, a contradiction. Therefore, $y_{12}\in I$.
\end{proof}

\begin{lemma}
\label{lem6.19}
Suppose $q^4\neq 1$.
Let $I$ be a Hopf ideal in $W$. Then one of the following must occur:
$$I\subseteq J^2, \quad e_{12}\in I, \quad e_{21}\in I,$$
where $J$ denotes the Jacobson radical of $W$.
\end{lemma}

\begin{proof} 
Suppose that $I \nsubseteq J^2$. 
By the $k$-linear basis of $W$ 
given in Lemma \ref{lem6.17}(a), we have $(e_{12}) \cap (e_{21})
=(e_{12}e_{21})\subseteq J^2$. So we can assume, without loss 
of generality, that $I\not\subseteq (e_{21})$. This means that 
$I+(e_{21})$ is a nonzero Hopf ideal of $B:=W/(e_{21})$.
By Lemma~\ref{lem6.18}(c), $e_{12} \in I + \left( e_{21} \right)$ that
is, $e_{12} + e_{21} p \in I$ for some $p \in W$. If $p = 0$, then we are 
done. If not, we apply Lemma~\ref{lem6.17}(c) to obtain 
$e_{12} + e_{21} p_0 =(e_{12}+e_{21}p)_1\in I$
where $p_0$ is the degree $0$ component of $p$. 

Next we use the hypothesis $q^4\neq 1$.
Since $p_0$ is generated by $e_{11}$, 
$S^2(p_0) = p_0$, so
\begin{eqnarray*}
    S^2 \left( e_{12} + e_{21} p_0 \right) & = & q^{2} e_{12} + q^{-2} e_{21}
    p_0 .
\end{eqnarray*}
Since $q^4 \neq 1$, we have that $e_{12} 
=\frac{1}{q^{2}-q^{-2}} (S^2(e_{12}+e_{21}p_0)
-q^{-2}(e_{12}+e_{21}p_0))\in I$. We are done.
\end{proof}

\begin{lemma}
  \label{x12x21}Let $\pi_{2}$ be the map in Diagram 5. If $q^{2} \neq 1$, then
  $\pi_{2} ( x_{12} ) =0$ if and only if $\pi_{2} ( x_{21} ) =0$. 
\end{lemma}

\begin{proof}
  Let $\{ p_{g}\}_{g \in \Gamma}$ denote the dual basis for $K' = ( k
  \Gamma )^{\circ}$. We will use the following fact repeatedly. If $x_{i j}
  p_{g} =0$, then
  \begin{eqnarray*}
    x_{i1} p_{h} \otimes x_{1j} p_{h'} & = & 0\\
    x_{i2} p_{h} \otimes x_{2j} p_{h'} & = & 0
  \end{eqnarray*}
  for all $h,h' \in G$ such that $h h' =g$. Indeed, the terms on the left
  above sum to $\Delta ( x_{i j} p_{g} ) ( p_{h} \otimes p_{h'} ) =0$
  and are eigenvectors of $S^{2} \otimes 1$ with different eigenvalues: $1$,
  $q^{2}$ for $i=1$; and $q^{-2}$, 1 for $i=2$.
  
  Suppose that $\pi_{2} ( x_{12} ) =0$. By exactness, we have that $x_{12} \in K(K')^+$, so we can write $x_{12} = \sum_{g \in G \setminus \{1\}} y_q p_g$, where $y_g \in K$. Thus, $x_{12}p_1 =0$. We show below that
  $x_{21} p_{1} =0$, which is equivalent to $\pi_{2} ( x_{21} ) =0$. The
  converse follows by symmetry.
  
  If $x_{12} =0$, then $x_{11}$ is grouplike and $K$ is
  generated by $x_{11}$, $x_{11}^{-1} x_{21}$. Since $x_{11}^{-1} x_{21}$ is
  skew-primitive, we see that $K$ is pointed, which is a contradiction. This
  shows that $x_{12} \neq 0$, so
 we must have $x_{12} p_{g} \neq 0$ for some $g \in
  \Gamma$. From $x_{12} p_{1} =0$, we get $x_{12} p_{g} \otimes x_{22}
  p_{g^{-1}} =0$ which implies $x_{22} p_{g^{-1}} =0$. Using this we get
  $x_{21} p_{g^{-2}} \otimes x_{12} p_{g} =0$ so $x_{21} p_{g^{-2}}
  =0$.
  
  Now if $x_{21} p_{h} =0$ for some $h \in \Gamma$ (for example, $h = g^{-2}$), then $x_{22} p_{h}
  \otimes x_{21} p_{1} =0$. We have $x_{21} p_{1} =0$ or $x_{22}
  p_{h} =0$. In the first case, we are done, so assume $x_{22} p_{h}
  =0$. This gives $x_{21} p_{h g^{-1}} \otimes x_{12} p_{g} =0$ which
  implies $x_{21} p_{h g^{-1}} =0$.
  
  The last two paragraphs show that $x_{21} p_{g^{-i}} =0$ for $i
  \geqslant 2$. Since $\Gamma$ is finite, this shows that $x_{21} p_{1} =0$
  which completes the proof.
\end{proof}

\begin{corollary}\label{newcor}
  If $q^{4} \neq 1$, then the Hopf ideal $I$ of $\mathfrak{u}_{2,q} (
  \mathfrak{s l}_{2} )^{\circ}$ appearing in Diagram 5 is $(0)$.
\end{corollary}
\begin{proof}
  First note that $q^{4} \neq 1$ implies $q^{2} \neq 1$, so we can use Lemmas
  \ref{lem6.19} and \ref{x12x21}. If $e_{12} \in I$, then $\pi_{2} ( x_{12} )
  =0$. By Lemma \ref{x12x21}, we also have $\pi_{2} ( x_{21} ) =0$, so $e_{21}
  \in I$. This shows there are really only two cases in Lemma \ref{lem6.19},
  namely $I \subseteq J^{2}$ or $e_{12} \in I$. We show that $I=0$ in the
  first case, and the second case leads to a contradiction.
  
  Assume $I \subseteq J^{2}$. Let $\left\{ C_j \right\}$ denote the coradical filtration of
  $\mf{u}_{2,q}(\mf{sl}_2)$. By
  \cite[Proposition 5.2.9]{Montgomery}, $C_s = \left( J^{s + 1} \right)^{\perp}$, so we have that 
$I^{\perp} \supseteq \left( J^2\right)^{\perp} = C_1$. By \cite[Section~5]{Takeuchi},    
we see that $C_1$ generates
$\mf{u}_{2,q}(\mf{sl}_2)$, hence $I = 0$.
  
  Assume $e_{12} \in I$. As we observed at the beginning of the proof, this
  implies $\pi_{2} ( x_{12} ) = \pi_{2} ( x_{21} ) =0$. Then
  $\mathfrak{u}_{2,q} ( \mathfrak{s l}_{2} )^{\circ} /I \simeq k F$ for some
  finite cyclic group $F$, since it is generated by a single grouplike element
  $\pi_{2} ( x_{11} )$. Therefore $K$ is a Hopf algebra extension
  \[ k \longrightarrow ( k  \Gamma )^{\circ} \longrightarrow K \longrightarrow
     k F \longrightarrow k. \]
By \cite[Proposition 1.5]{masuokaMSRI} (or \cite[page~571]{masuoka:calculations}), the algebra $K$ is a crossed product. So, by  \cite[Proposition~7.4.2(2)]{Montgomery} it is
  semisimple. This is a contradiction. 
 
\end{proof}

Therefore, we have proved the following proposition.

\begin{proposition} \label{pro6.25}
Assume Hypothesis \ref{hyp0.2} with $H$ non-semisimple and the left
$H$-module $U$ is simple. Here, $H^{\circ}$ is a finite Hopf algebra
quotient of $\mc{O}_q(SL_2(k))$ that coacts on $k_{q}[u,v]$ for
$q$ a root of unity with $q^4 \neq 1$.

Assume that the $\mathrm{ord}(q) =: l=2m$ is even. Then, $H^{\circ}$ fits
into the following exact sequence of Hopf algebras:
\begin{equation} \label{E6.26.1} \tag{E6.26.1}
k \rightarrow (k\Gamma)^{\circ} \rightarrow H^{\circ} \rightarrow
\mf{u}_{2,q}(\mf{sl}_2)^{\circ} \rightarrow k,
\end{equation}
for $\Gamma$ a finite subgroup of $PSL_2(k)$.
Moreover, $\dim_k H = \dim_k H^{\circ} =
2|\Gamma|m^3=2|\Gamma|(l/2)^3$.  \qed
\end{proposition}

Now, we study the invariant subring $R^H \cong R^{\mathrm{co}\; H^{\circ}}$ for
the pair $(H,R)$ in the proposition above. By (\ref{E6.26.1}),
$R^{\mathrm{co}\; H^{\circ}} 
= (R^{\mathrm{co}\; \mf{u}_{2,q}(\mf{sl}_2)^{\circ}})^{\mathrm{co}\;
(k\Gamma)^{\circ}}$ for $R = k_{q}[u,v]$ with $q$ a root of
unity for $q^4 \neq 1$.

\begin{lemma}
\label{lem6.26} The coinvariant subring $R^{\mathrm{co}\;
\mf{u}_{2,q}(\mf{sl}_2)^{\circ}}$ of $R=k_{q}[u,v]$ is
isomorphic to $k[a,b,c]/(b^2-ac)$ where $a=u^{2m}$, $b=u^mv^m$,
$c=v^{2m}$. Here, the coaction has trivial homological
codeterminant.
\end{lemma}

\begin{proof}
By the quantum binomial theorem, $u^{2m}, u^mv^m, v^{2m} \in  R^{\mathrm{co}\;
\mf{u}_{2,q}(\mf{sl}_2)^{\circ}}$.  Consider the factor ${\mathfrak
u}_{2,q}(\mf{sl}_2)^{\circ}/(e_{21})$, which is isomorphic to the
Hopf algebra $T_{q,0,2m}$ in Definition \ref{def6.4}. Now by Lemma
\ref{lem6.7},
$$R^{\mathrm{co}\; \mf{u}_{2,q}(\mf{sl}_2)^{\circ}} \subseteq
R^{\mathrm{co}\; T_{q,0,2m}} \subseteq k[a,b,c]/(b^2-ac)$$ where
$a=u^{2m}$, $b=u^mv^m$, $c=v^{2m}.$
So the assertion holds.
\end{proof}

Now we deal with the case where $q^4=1$.

\begin{remark}
\label{rem6.24}
If $q^4 = 1$, then the Hopf ideal $I$ of
$W = \mf{u}_{2,q}(\mf{sl}_2)^{\circ}$ appearing in Diagram 5 is equal to
  $J_{\lambda} = ( e_{12} - \lambda e_{21} e_{11}^{2} )$ for some $\lambda \in
  k^{\ast}$ or $( 0 )$. By using similar arguments to the $q^4 \neq 1$ case, it is easy to check that $W/J_\lambda \cong W/J_{1}$ for all $\lambda \neq 0$.
So under the hypothesis of $U$ being simple, there are two possibilities for 
$\mf{u}_{2,q}(\mf{sl}_2)^{\circ}/I$ in Diagram 5, either 
$\mf{u}_{2,q}(\mf{sl}_2)^{\circ}$ 
or $\mf{u}_{2,q}(\mf{sl}_2)^{\circ}/(e_{12}-e_{21}e_{11}^2)$.
Note that the Hopf algebra 
$\mf{u}_{2,q}(\mf{sl}_2)^{\circ}/(e_{12}-e_{21}e_{11}^2)$ 
is the unique $8$-dimensional non-semisimple
non-pointed Hopf algebra \cite[Theorem 3.5]{Stefan}.
\end{remark}

\begin{lemma}\label{lem6.27}
  If $q^4 = 1$, then $T \assign \mathfrak{u}_{2, q} \left( \mathfrak{s l}_2
  \right)^{\circ} / \left( e_{12} - e_{21} e_{11}^2 \right)$ coacts on $R =
  k_q \left[ u, v \right]$. In this case,
  $R^{\mathrm{co} \; T} = R^{\mathrm{co} \; \mathfrak{u}_{2, q} 
\left( \mathfrak{s
  l}_2 \right)^{\circ}}$.
\end{lemma}

\begin{proof}
  The Hopf algebra quotients $\mathfrak{u}_{2, q} \left( \mathfrak{s l}_2
  \right)^{\circ} \to T$ and $T \to T / \left( e_{12},
  e_{21} \right)$ induce the following chain
  $$R^{\mathrm{co}~
  \mathfrak{u}_{2, q} \left( \mathfrak{sl}_2 \right)^{\circ}} \subseteq
  R^{\mathrm{co} \; T} \subseteq R^{\mathrm{co}(T / \left( e_{12}, e_{21}
  \right))}$$ of coinvariant subrings. Now $T / \left( e_{12}, e_{21} \right)
  \simeq kC_4$ coacts diagonally on $R$, so 
$R^{\mathrm{co}(T/( e_{12}, e_{21}))} = R^{(4)}$, the fourth Veronese 
subalgebra of $R$. Therefore, it suffices to show that for each 
$f \in R^{(4)}$ coinvariant under the $T$-coaction, we have 
$f \in R^{\mathrm{co} \; \mathfrak{u}_{2, q}( \mathfrak{sl}_2)^{\circ}}$. 
Suppose that $f \in R^{(4)}$ and $f$ is $T$-coinvariant. By 
Lemma~\ref{lem6.26}, $R^{\left( 4 \right)}$ is generated by 
$1, uv^3$ and $u^3 v$ as an $R^{\mathrm{co} \; \mathfrak{u}_{2, q} 
(\mathfrak{sl}_2)^{\circ}}$-module, so we can write $f = f_0 + f_1
  uv^3 + f_2 u^3 v$ for some $f_i \in R^{\mathrm{co} \;
  \mathfrak{u}_{2, q} \left( \mathfrak{sl}_2 \right)^{\circ}}$. Since $f$ is
  $T$-coinvariant, so is $f - f_0$, that is, $\rho \left( f - f_0 \right) =
  \left( f - f_0 \right) \otimes 1 = \left( f_1 uv^3 + f_2 u^3 v \right)
  \otimes 1$. Now
  \begin{eqnarray*}
    \rho( f - f_0 ) & = & \left( f_1 \otimes 1 \right) \rho(
    uv) \rho( v^2 ) + \left( f_2 \otimes 1 \right) \rho( u^2 ) \rho ( uv )\\
    & = & - \left( f_1 \otimes 1 \right) \rho( v^2) \rho (
    uv ) + \left( f_2 \otimes 1 \right) \rho ( u^2 ) \rho
    ( uv )\\
    & = & \left( \left( f_2 u^2 - f_1 v^2 \right) \otimes e_{11}^2 \right) 
    \left( u^2 \otimes e_{11} e_{12} + v^2 \otimes e_{21} e_{22} + uv \otimes
    1 \right)\\
    & = &   \left( f_2 u^2 - f_1 v^2 \right) u^2 \otimes e_{11}^3
    e_{12} + q^{- 1} \left( f_2 u^2 - f_1 v^2 \right) v^2 \otimes e_{11}
    e_{21} + \left( f_2 u^2 - f_1 v^2 \right) uv \otimes e_{11}^2 .
  \end{eqnarray*}
  Since $\{e_{11}^3 e_{12}, ~e_{11} e_{21},~ e_{11}^2, ~1\}$ are linearly
  independent, we have that $f_1 u v^3 + f_2 u^3 v = 0$. Hence $f = f_0 \in
  R^{\mathrm{co} \; \mathfrak{u}_{2, q} \left( \mathfrak{sl}_2
  \right)^{\circ}}$. This completes the proof.
\end{proof}

\begin{proposition} \label{pro6.28}
Let $(H,R)$ be a pair satisfying Hypothesis \ref{hyp0.2} in the
setting of Theorem \ref{thm6.2}(c3). Then, the invariant subring
$R^H$ of $R$ is (AS-)Gorenstein.
\end{proposition}

\begin{proof}
  Recall that $H^{\circ}$ fits into an exact sequence $k \to (k
  \Gamma)^{\circ} \to H^{\circ} \to \mathfrak{u}_{2, q} \left( \mathfrak{s
  l}_2 \right)^{\circ} \to k$ of Hopf algebras, so
  \[ R^H = R^{\mathrm{co} \; H^{\circ}} = \left( R^{\mathrm{co}
     \; \mathfrak{u}_{2, q} \left( \mathfrak{s l}_2
     \right)^{\circ}} \right)^{\mathrm{co} \; (k \Gamma)^{\circ}}
     . \]
  If $q^4 = 1$, then $H^{\circ}$ can also fit into an exact sequence $k \to (k
  \Gamma)^{\circ} \to H^{\circ} \to T \to k$; see Lemma~ \ref{lem6.27}. 
Lemma~\ref{lem6.27} also shows that
  \[ R^H = R^{\mathrm{co} \; H^{\circ}} = \left( R^{\mathrm{co}
     \; T} \right)^{\mathrm{co} \; (k
     \Gamma)^{\circ}} = \left( R^{\mathrm{co} \; \mathfrak{u}_{2,
     q} \left( \mathfrak{s l}_2 \right)^{\circ}} \right)^{\mathrm{co}
     \; (k \Gamma)^{\circ}} . \]
  In both cases, we have 
  $R^{\mathrm{co} \; \mathfrak{u}_{2, q} 
\left( \mathfrak{s l}_2 \right)^{\circ}} \cong 
k \left[ u^{2 m}, v^{2 m}, u^m v^m \right]$ 
  with induced coaction of $(k\Gamma)^{\circ}$ given by
\begin{eqnarray*}
  \bar{\rho} \left( u^{2 m} \right) & = & u^{2 m} 
\otimes e_{11}^{2 m} + v^{2 m}
  \otimes e_{21}^{2 m} + 2 u v \otimes e_{11}^m e_{21}^m\\
  \bar{\rho} \left( v^{2 m} \right) & = & u^{2 m} 
\otimes e_{12}^{2 m} + v^{2 m}
  \otimes e_{22}^{2 m} + 2 u v \otimes e_{12}^m e_{22}^m\\
  \bar{\rho} \left( u^m v^m \right) & = & u^{2 m} 
\otimes e_{11}^m e_{12}^m + v^{2
  m} \otimes e_{21}^m e_{22}^m + u v \otimes 
\left( e_{11}^m e_{22}^m + \left(
  - 1 \right)^m e_{21}^m e_{12}^m \right).
\end{eqnarray*}

If $m$ is even (resp. odd), then the above coaction of $\left( k \Gamma
\right)^{\circ}$ on $k \left[ u^{2 m}, v^{2 m}, u^m v^m \right]$ coincides
with the coaction given in Lemma~\ref{lem6.29} below for $\nu = 1$ 
(resp. $\nu = - 1$). 
By Lemma~\ref{lem6.29}(b), we conclude that 
$R^H \simeq k \left[ u^{2 m}, v^{2 m}, u^m
v^m \right]^{\mathrm{co} \; \left( k \Gamma \right)^{\circ}}$ is AS Gorenstein.
\end{proof}

\begin{lemma}\label{lem6.29}
Consider the coaction of $\mathcal{O}(PSL_2 (k))\cong 
k [\alpha_{i j} \alpha_{k l} \mid i, j, s, t =1, 2]$ on 
$S = k[ x^2, y^2, x y]$ below:
\begin{eqnarray*}
\tau_{\nu} \left( x^2 \right) & = & x^2 \otimes \alpha_{11}^2 + y^2
\otimes \alpha_{21}^2 + 2 xy \otimes \alpha_{11} \alpha_{21}\\
\tau_{\nu} \left( y^2 \right) & = & x^2 \otimes \alpha_{12}^2 + y^2
\otimes \alpha_{22}^2 + 2 xy \otimes \alpha_{12} \alpha_{22}\\
\tau_{\nu} \left( xy \right) & = & x^2 \otimes \alpha_{11} \alpha_{12} +
y^2 \otimes \alpha_{21} \alpha_{22} + xy \otimes \left( \alpha_{11}
\alpha_{22} + \nu \alpha_{21} \alpha_{12} \right)
\end{eqnarray*}
where $\nu = \pm 1$.
\begin{enumerate}
\item[(a)] Let $\tau'_{\nu}$ be the coaction of 
$\mathcal{O}_{\nu} (SL_2 (k))=k \langle \alpha^{(\nu)}_{i j} \mid 
i, j = 1, 2 \rangle$ on $k_{\nu} [x,y]$ given by
\[ \tau'_{\nu} \left( x \right) = x \otimes \alpha_{11}^{\left( \nu
       \right)} + y \otimes \alpha_{21}^{\left( \nu \right)} \hspace{1em}
       \text{and} \hspace{1em} \tau'_{\nu} \left( y \right) = x \otimes
       \alpha_{12}^{\left( \nu \right)} + y \otimes \alpha_{22}^{\left( \nu
       \right)} . \]
The induced coaction of $\mathcal{O} (P SL_2 (k)) = k \left[
\alpha^{\left( \nu \right)}_{i j} \alpha_{k l}^{\left( \nu \right)} \mid
i, j, k, l = 1, 2 \right]$ on $k_{\nu} \left[ x, y \right]^{\left( 2
\right)} \simeq S$ is equal to $\tau_{\nu}$.
    
\item[(b)] Let $\Gamma$ be a finite subgroup of $P S L_2 \left( k \right)$, 
so that $\tau_{\nu}$ induces a coaction of $\left( k \Gamma \right)^{\circ}$
on $S$. Then $S^{\mathrm{co} \; \left( k \Gamma \right)^{\circ}} = k_{\nu}
\left[ x, y \right]^{\mathrm{co}\;\tilde{H}}$ where $\tilde{H}$ is a Hopf
algebra extension of $k\mathbb{Z}_2$ by $(k\Gamma)^{\circ}$.
In particular, the $\tilde{H}$-coaction on $k_{\nu} \left[ x, y \right]$
has trivial homological codeterminant, so $S^{\mathrm{co} \; 
(k \Gamma )^{\circ}}$ is AS Gorenstein.
\end{enumerate}
\end{lemma}

\begin{proof}
Part (a) follows from direct computation. For part (b), we assume 
$\nu = - 1$. The argument for $\nu = 1$ is similar and will be omitted. 
Now $\left( k \Gamma  \right)^{\circ}$ is a finite dimensional Hopf 
algebra quotient $\mathcal{O}  \left( P S L_2 \right) / I$. Since 
$I \subset \mathcal{O}(PSL_2) \subset \mathcal{O}_{- 1}(SL_2)$, we can 
consider the Hopf ideal $\tilde{I}$ of $\mathcal{O} \left( S L_2 \right)$ 
generated by $I$. Then $\mathcal{O}_{- 1} \left( S L_2 \right) / \tilde{I} 
\assign \tilde{H}$ is a finite dimensional Hopf algebra quotient of 
$\mathcal{O}_{-1} \left( S L_2 \right)$ such that
  \[ k \longrightarrow \left( k \Gamma \right)^{\circ} \longrightarrow
     \tilde{H} \longrightarrow k\mathbb{Z}_2 \longrightarrow k \]
so $k_{- 1} \left[ x, y \right]^{\mathrm{co} \; \tilde{H}} = \left( k_{- 1}
\left[ x, y \right]^{\mathbb{Z}_2} \right)^{\mathrm{co} \; \left( k \Gamma
\right)^{\circ}} = S^{\mathrm{co} \; \left( k \Gamma \right)^{\circ}}$. 
Using the quantum determinant relation, we have
\begin{eqnarray*}
    \tau_{- 1} \left( x y + y x \right) = x y \otimes \left( \alpha_{11}
    \alpha_{22} + \alpha_{12} \alpha_{21} \right) + y x \otimes \left(
    \alpha_{21} \alpha_{12} + \alpha_{22} \alpha_{11} \right) 
= \left(x y+y x\right)\otimes 1 .
  \end{eqnarray*}
so by Theorem \ref{thm2.1}, the $\tilde{H}$-coaction has trivial
homological codeterminant. Finally by \cite[Theorem 0.1]{KKZ:Gorenstein} 
$S^{\mathrm{co} \; \left( k \Gamma
\right)^{\circ}} (\cong k_{-1}[x,y]^{\mathrm{co}\; \tilde{H}})$ is 
AS Gorenstein. 
\end{proof}


\section{McKay quivers}
\label{sec:McKay}

Let $H$ be a semisimple Hopf algebra and $U$ be a distinguished $H$-module.
The {\em{(left) McKay quiver}} $Q ( H,U )$ is the quiver whose vertices are
indexed by the isomorphism classes of irreducible $H$-modules $\{ S_{i} \}$
with $m_{i j} \assign \Hom ( S_{i} ,U \otimes S_{j} )$ arrows
from $S_{i}$ to $S_{j}$. Suppose $R=k \langle U \rangle / ( r )$ is an
$H$-module algebra as in Theorem~\ref{thm0.4}, then we associate to $( H,R )$ the McKay
quiver $Q ( H,U )$.

\begin{proposition}
\label{pro7.1} Let $H$ be a semisimple Hopf algebra acting on an AS regular algebra $R = k
\langle U \rangle/(r)$ of global dimension two under Hypotheses~\ref{hyp0.3}. Then the McKay
quiver $Q(H,U)$ is one of the types as listed in Table 6 below. Refer to Theorems 4.5 and 5.2 for the description of $H$ in the cases (a1) - (b4).
\end{proposition}

\noindent Here, $\tilde{\Gamma}$ denotes a nonabelian finite subgroup of $SL_2(k)$.

\[
\begin{array}{|rlccc|cc|}
\hline
&\text{$H$} & &&& & Q(H,U) \text{~is of type:}\\
\hline
&&&&&&\\
\text{(a1)~}& k\tilde{\Gamma}& &\tilde{\Gamma}
\text{~ of type~ D$_n$, E$_6$, E$_7$, E$_8$}, \text{~resp.} &&& \text{$\widetilde{\text{D}}_n$, $\widetilde{\text{E}}_6$, $\widetilde{\text{E}}_7$, $\widetilde{\text{E}}_8$}
\text{~resp.}\\
\text{(a2)~}& kD_{2n}&&&&& \begin{cases}
              \text{$\widetilde{\text{D}}_{\frac{n+4}{2}}$},  \text{~if~} n \text{~even}\\
              \text{$\widetilde{\text{DL}}_{\frac{n+1}{2}}$}, \text{~if~} n \text{~odd}
\end{cases}\\
\text{(a3)~}& \mc{D}(\tilde{\Gamma})^{\circ}  & &
\tilde{\Gamma} \text{~of type~}       \text{D$_n$, E$_6$, E$_7$, E$_8$},  \text{~resp.}
&&& \text{$\widetilde{\text{D}}_n$, $\widetilde{\text{E}}_6$, $\widetilde{\text{E}}_7$, $\widetilde{\text{E}}_8$} \text{~resp.}\\
\text{(b1)~} & kC_2 &&&&& \text{$\widetilde{\text{A}}$}_1\\
\text{(b2)~} & kC_2 &&&&& \text{$\widetilde{\text{L}}$}_1\\
\text{(b3)~} & kC_n & &n\geq 3&&& \text{$\widetilde{\text{A}}$}_{n-1}\\
\text{(b4)~} & (kD_{2n})^{\circ} &&&&& \text{$\widetilde{\text{A}}$}_{2n-1}\\
\hline
\end{array}
\]
\medskip

\begin{center} Table 6: McKay Quivers for $(H,R)$ in
Theorem \ref{thm0.4} with $H$ semisimple \end{center}

\bigskip

\begin{proof}
Cases (a1), (b1), and (b3) are well known. 
Case (b4) follows from direct computation.
For case (a2), since $U$ is a faithful irreducible 2-dimensional $D_{2n}$-module,
the McKay quiver $Q(k D_{2n},U)$ can be found in \cite[pages 320 and~325]{HPR}. 

For case (b2), let $S_{1}$ and $S_{-1}$ be the trivial and sign representations of $C_2$, respectively, so that $U=S_{1}\oplus S_{-1}$.
Then, $S_{\pm1}\otimes U= S_{1} \oplus S_{-1}$. Hence $Q(kC_2, U)$ is of type $\widetilde{\text{L}}$ as pictured below.

Finally in case (a3), dualizing the result of \cite[Lemma 5.15]{BichonNatale} gives an algebra isomorphism $H\cong k\tilde{\Gamma}$,
where $\tilde{\Gamma}$ is a non-abelian binary dihedral group. Therefore, the irreducible representations of $H$ and $\tilde{\Gamma}$ 
have the same dimensions. By basic combinatorial considerations (alternatively, a character theoretic argument is given in 
the proof of \cite[Remark 5.23]{BichonNatale}), the McKay quiver of $\mathcal{D}(\tilde{\Gamma})^{\circ}$ coincides with that of 
$\tilde{\Gamma}$ if $\tilde{\Gamma}$ is exceptional. The binary dihedral case follows from \cite[Proposition 3.9]{Masuoka:cocycle}. 

\end{proof}

\begin{remark} \label{rem7.2}
Note that all of the McKay quivers $Q(H,U)$ in the proposition above are of Dynkin type $\widetilde{\text{A}}$-$\widetilde{\text{D}}$-$\widetilde{\text{E}}$,
except in cases (b2) and (a2) for $n$ odd.  In the latter case, $Q$ is of
type $\widetilde{\text{DL}}$. 


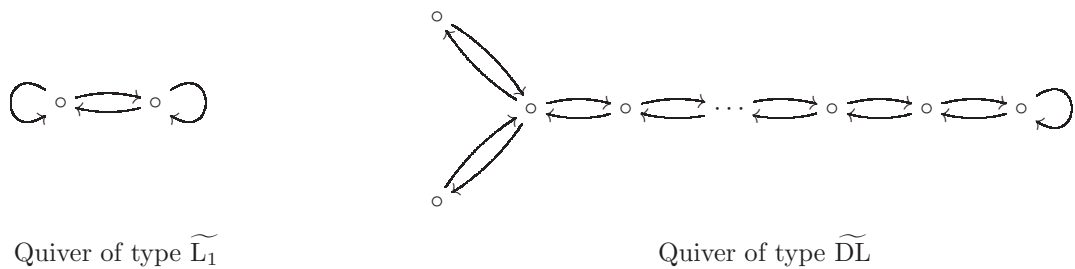
\vspace{.4in}
\[
\hspace{-4.5in}
\xymatrix{
\circ  \ar@/^{.3pc}/@{->}[r] \ar@{->}@(ul,dl)
&\circ \ar@/^{.3pc}/@{->}[l] \ar@{->}@(ur,dr)
}
\]
\vspace{-.95in}
\[
\hspace{2in}
\xymatrix{
\circ \ar@/^{.3pc}/@{->}[dr]\\
 &\circ \ar@/^{.3pc}/@{->}[ul]  \ar@/^{.3pc}/@{->}[dl]   \ar@/^{.3pc}/@{->}[r] &\circ  \ar@/^{.3pc}/@{->}[l] \ar@/^{.3pc}/@{->}[r] &\cdots  \ar@/^{.3pc}/@{->}[l] \ar@/^{.3pc}/@{->}[r]
&\circ  \ar@/^{.3pc}/@{->}[l] \ar@/^{.3pc}/@{->}[r]&\circ   \ar@/^{.3pc}/@{->}[l] \ar@/^{.3pc}/@{->}[r]&\circ \ar@/^{.3pc}/@{->}[l]\ar@{->}@(ur,dr)\\
\circ \ar@/^{.3pc}/@{->}[ur]
}
\]
\begin{center} \hspace{-1in}Quiver of type $\widetilde{\text{L}_1}$\hspace{2.3in}Quiver of type $\widetilde{\text{DL}}$ \end{center}
\medskip
\medskip

The quiver of type $\widetilde{\text{DL}}$ also arises in other studies of a quantum
McKay correspondence. For instance, see the work of Malkin, Ostrik,
and Vybornov \cite{MOV} on quiver varieties; there, type $\widetilde{\text{DL}}$ is
referred to as type T.
\end{remark}


\section{Questions for further study}
\label{sec8}

We conclude by listing several open questions that merit further study. First, one natural extension of our main classification result, Theorem \ref{thm0.4}, is to obtain a similar
classification result for Artin-Schelter regular algebras of global dimension $3$.
These algebras were classified by Artin, Schelter, Tate and Van den Bergh,
\cite{ArtinSchelter, ATV1, ATV2}. It is likely that new methods beyond
those developed in this paper are needed for this problem. We state this task as follows. 

\begin{question}
\label{que0.5}
Let $A$ be an Artin-Schelter regular algebra of global dimension $3$. What are the
finite dimensional Hopf algebras that act inner faithfully on and preserve the grading of $A$ with trivial homological determinant?
\end{question}

Moreover, recall Proposition \ref{pro0.5}: we have that $R^H$ is AS Gorenstein, for all $H$-actions on $R$ in Theorem \ref{thm0.4}. This suggests the following question pertaining to
Hopf actions on AS regular algebras of global dimension $n$. Recall that Hopf coaction is dual to Hopf action, that is to say, a Hopf algebra $H$
acts on an algebra $R$ from the left if and only if  its dual Hopf algebra
$K:=H^{\circ}$ coacts on $R$ from the right. 

\begin{question}
\label{que0.7}
Let $K$ be a finite dimensional Hopf algebra quotient of the standard
quantum special linear group ${\mathcal O}_q(SL_n(k))$ that coacts on the
skew polynomial ring $k_q[v_1,\cdots,v_n]$ naturally. Is the
co-invariant subring $k_q[v_1,\cdots,v_n]^{\mathrm{co} K}$ Artin-Schelter Gorenstein?
\end{question}

Observe that we assume that all $H$-actions on $R$ in this work have trivial homological determinant. This begs the question of whether a classification result could be obtained for $H$-actions on $R$ with {\it arbitrary} homological determinant. Naturally, this would be a quantum analogue of such a result  for actions of finite subgroups of $GL_2(k)$ on $k[u,v]$ in classical invariant theory. This prompts several questions, some of which are given below.

\begin{question} Consider the following questions.
\begin{enumerate}
\item  What are the finite dimensional Hopf algebras $H$ that act on an Artin-Schelter regular algebra $R$ (of global dimension two), with the action having arbitrary homological determinant? 
\item For the actions in part (a), when are the corresponding invariant rings $R^H$ Artin-Schelter regular? When are they Artin-Schelter Gorenstein? 
\end{enumerate}
\end{question}

More generally for the first question in part (b), we ask if there is a  version of the Shephard-Todd-Chevalley theorem for finite dimensional Hopf actions on AS regular algebras. This task has been addressed for finite group actions on skew polynomial rings \cite{KKZ:STC}.

On the other hand, in the view of a noncommutative McKay correspondence in our context, there are many questions that can be posed. For example,
it would be useful to have an analogue of the following theorem of Auslander.

\begin{theorem} (Auslander) \cite[Theorems~5.15]{LeuschkeWiegand} 
Let $G$ be a finite subgroup of $SL_n(k)$ and let $S$ be the polynomial ring 
$k[v_1, \dots, v_n]$. Then, $S \# G$ is isomorphic to $\End_{S^G}(S)$ as rings.
\end{theorem} 

This result was vital for establishing an equivalence between the category 
of finitely generated projective modules over $k[u,v] \#G$ and the 
category of $k[u,v]^G$-direct summands of $k[u,v]$, for a finite 
subgroup $G$ of $SL_2(k)$. Note that the objects of the latter 
category are precisely the maximal Cohen-Macaulay modules over $k[u,v]^G$. Thus, we have the question below.

\begin{question}
Given a finite dimensional Hopf algebra $H$ acting on an Artin-Schelter 
regular algebra $R$, satisfying Hypothesis~\ref{hyp0.3}, do we have 
that $R \# H$ is isomorphic to $\End_{R^H}(R)$ as rings?
\end{question}


\section*{Acknowledgments}
The authors are grateful to the referee for pointing out several typographical errors and for making suggestions that improved greatly the exposition of this manuscript. The authors also thank Jacques Alev, Jim Kuzmanovich, and Graham Leuschke for several useful discussions and valuable comments. We thank especially Michael Wemyss for pointing out corrections for Section~\ref{sec:McKay}. C. Walton and J.J. Zhang were supported by the US National
Science Foundation: NSF grants DMS-1102548 and DMS-0855743, respectively.  E. Kirkman was partially supported by grant \#208314 from the Simons Foundation.

\bibliography{HopfAS2trivHdet_biblio}

\end{document}